\documentclass[11pt,letterpaper,reqno]{amsart} 

\usepackage[margin=1.5in]{geometry} 
\usepackage{graphicx} 

\usepackage{dsfont} 

\usepackage{booktabs} 
\usepackage{array} 
\usepackage{paralist} 
\usepackage{verbatim} 
\usepackage{mathrsfs}
\usepackage{amssymb}
\usepackage{amsthm}
\usepackage{lipsum}
\usepackage{amsmath,amsfonts,amssymb}
\usepackage{esint}
\usepackage{graphics}
\usepackage{enumerate}
\usepackage{mathtools}
\usepackage{xfrac}
\usepackage{subcaption}
\usepackage{stmaryrd}
 \usepackage{mathabx}

\usepackage[dvipsnames]{xcolor}
\usepackage[colorlinks=true, pdfstartview=FitV, linkcolor=blue, citecolor=blue, urlcolor=blue]{hyperref}
\usepackage[normalem]{ulem}

\usepackage{tikz}
\usetikzlibrary{calc}
\usepackage{pgf}
\usetikzlibrary{external}
\numberwithin{equation}{section}
\numberwithin{figure}{section}

\newtheorem{theorem}{Theorem}[section]

\newtheorem{assumption}[theorem]{Assumption}
\newtheorem{corollary}[theorem]{Corollary}
\newtheorem{proposition}[theorem]{Proposition}
\newtheorem{lemma}[theorem]{Lemma}
\theoremstyle{definition}
\newtheorem{definition}[theorem]{Definition}

\newtheorem{remark}[theorem]{Remark}

\newcommand*{\N}{\ensuremath{\mathbb{N}}}

\newcommand*{\Z}{\ensuremath{\mathbb{Z}}}

\newcommand*{\R}{\ensuremath{\mathbb{R}}}

\renewcommand*{\tilde}{\widetilde}

\DeclareSymbolFont{boldoperators}{OT1}{cmr}{bx}{n}
\SetSymbolFont{boldoperators}{bold}{OT1}{cmr}{bx}{n}
\usepackage{accents}

\newcommand{\T}{\mathbb{T}}

\def\XXint#1#2#3{{\setbox0=\hbox{$#1{#2#3}{\int}$}
\vcenter{\hbox{$#2#3$}}\kern-.5\wd0}}

\renewcommand{\phi}{\varphi}
\newcommand{\indc}{\mathds{1}}

\renewcommand{\>}{\rangle}
\newcommand{\C}{\mathbb{C}}

\newcommand{\sol}{\mathcal{T}}
\newcommand{\E}{\mathbb{E}}

\renewcommand{\hat}{\widehat}

\makeatletter
\pgfmathdeclarefunction{erf}{1}{%
  \begingroup
    \pgfmathparse{#1 > 0 ? 1 : -1}%
    \edef\sign{\pgfmathresult}%
    \pgfmathparse{abs(#1)}%
    \edef\x{\pgfmathresult}%
    \pgfmathparse{1/(1+0.3275911*\x)}%
    \edef\t{\pgfmathresult}%
    \pgfmathparse{%
      1 - (((((1.061405429*\t -1.453152027)*\t) + 1.421413741)*\t 
      -0.284496736)*\t + 0.254829592)*\t*exp(-(\x*\x))}%
    \edef\y{\pgfmathresult}%
    \pgfmathparse{(\sign)*\y}%
    \pgfmath@smuggleone\pgfmathresult%
  \endgroup
}
\makeatother

\title[Obukhov--Corrsin spectrum of scalar turbulence]{The Obukhov--Corrsin spectrum of passive scalar turbulence through anomalous regularization}

\author{
Keefer Rowan
}\thanks{\'Ecole Polytechnique F\'ed\'erale de Lausanne.
{\footnotesize \href{mailto:keefer.rowan@epfl.ch}{keefer.rowan@epfl.ch}.}
}
\date{\today}

\begin{document}

\begin{abstract}
    The Obukhov--Corrsin spectrum predicts the distribution of Fourier mass for a passive scalar field advected by a ``turbulent'' velocity field with spatial regularity $C^\alpha_x$ for $\alpha \in (0,1)$ and subject to a time-stationary forcing. We prove the Obukhov--Corrsin spectrum holds after summing over geometric annuli in Fourier space---up to logarithmic corrections---as a consequence of a sharp anomalous regularization result. We then prove this anomalous regularization for a broad class of Kraichnan-type models. The proof of anomalous regularization relies on a Fourier space $\ell^p$ energy equality and a weighted lattice Poincar\'e inequality.
\end{abstract}

\maketitle
\section{Introduction}

Nearly a century ago, it became apparent that an exact description of a turbulent fluid was impossible and attention turned to understanding \textit{generic} and \textit{statistical} features of fluids~\cite{taylor_statistical_1935,frisch_turbulence_1995}. A pioneering and highly successful theory capturing the generic statistics of a turbulent fluid is Kolmogorov's K41 theory~\cite{kolmogorov_local_1941,kolmogorov_degeneration_1941,kolmogorov_dissipation_1941}. A central prediction of K41 theory is the equilibrium distribution of Fourier mass of a three-dimensional fluid velocity $u$ in statistical equilibrium under a steady forcing and vanishing diffusivity
\begin{equation}
\label{eq:kolmogorov 11 thirds}
\E_\mu |\hat u(k)|^2 \approx |k|^{-11/3},
\end{equation}
$k \in \Z^d$, $|k| \gg 1$, $\hat u(k)$ is the $k$th Fourier coefficient of $u : \T^d \to \R^d$, and $\mu$ is the ``vanishing-viscosity equilibrium measure''. 

Some aspects of the phenomenology of turbulent fluids have been given a rigorous mathematical foundation. Notably, the Onsager conjecture~\cite{onsager_statistical_1949} has seen an essentially complete resolution~\cite{constantin_onsagers_1994,buckmaster_anomalous_2015, isett_proof_2018,novack_intermittent_2023}. Additionally, the more robust Kolmogorov 4/5 law---also a central prediction of K41 theory---has been shown to necessarily arise in the equilibrium measure of fluid velocity given the physically well-motivated (but unproven) hypothesis of weak anomalous dissipation~\cite{bedrossian_sufficient_2019}. There are, however, numerous obstacles to giving rigorous proof to the prediction~\eqref{eq:kolmogorov 11 thirds}. There are physical obstacles:~\eqref{eq:kolmogorov 11 thirds} is, in fact, believed to be \textit{false}, due to the presence of \textit{intermittency corrections}---slightly increasing the exponent beyond $11/3$---further, there is not even a predicted universal value for this correction~\cite[Chapter 6]{frisch_turbulence_1995}. There are also profound mathematical obstacles: the three-dimensional viscous fluid equation is possibly not even globally well-posed and, beyond that, understanding the development of a turbulent cascade such as~\eqref{eq:kolmogorov 11 thirds} would require an extremely precise, uniform-in-viscosity control on the non-linear fluid equation which has essentially no precedent in the mathematical literature.

As such, we turn to a more tractable problem: \textit{passive scalar turbulence}. Passive scalars are scalar fields $\theta^\kappa_t : \T^d \to \R$ that are advected by a ``fluid-like'' velocity $v_t : \T^d \to \R^d$ but do not act on the velocity field. The scalar $\theta^\kappa_t : \T^d \to \R$ thus solves a linear advection-diffusion equation, and the velocity field is prescribed extrinsically. Ideally, we would take the velocity field $v_t$ to itself be a solution to a (stochastically forced) fluid equation, but that essentially returns us to the difficulty of the original problem. We therefore take \textit{synthetic velocity fields}, which are given explicit definition and for which we seek to retain various ``fluid-like'' properties. The simplest such property, which we always take in this work, is that the velocity field is incompressible: $\nabla \cdot v_t =0.$ Since we are trying to capture turbulent phenomena, we also take $v_t$ to have the regularity of a turbulent velocity field, which is (around) $C^{1/3}_x$; we in fact consider regularities $C^\alpha_x$ for $\alpha \in (0,1)$ as the phenomenology is similar in this parameter range.   

In this setting of passive scalar turbulence with spatially rough advecting flow, there is analogous phenomenological theory to K41, known as Obukhov--Corrsin theory~\cite{obukhov_structure_1949,corrsin_spectrum_1951}, which makes the prediction that for a $d$-dimensional passive scalar $\theta$ advected by a $C^\alpha_x$ ``turbulent'' velocity field and subject to a statistically steady forcing, we have in statistical equilibrium
\begin{equation}
\label{eq:obukhov corrsin}
\E_\mu |\hat \theta(k)|^2 \approx |k|^{-d- (1-\alpha)}.
\end{equation}
We note that---unlike in the fluid setting---we actually expect~\eqref{eq:obukhov corrsin} to hold; intermittency corrections are only expected to appear for moments greater than $2$~\cite{gawedzki_anomalous_1995,frisch_intermittency_1998,bernard_slow_1998,rosa_intermittency_2025}.

It is the goal of this work to show (a slightly weaker version of)~\eqref{eq:obukhov corrsin} holds for the invariant measure associated to any velocity field $v_t \in C^\alpha_x$ provided the properties of \textit{anomalous dissipation} and \textit{anomalous regularization} (up to the appropriate regularity index) hold. We then show that the desired anomalous dissipation and regularization hold for the stochastic version of the advection-diffusion equation given by taking the advecting flow to be a transport noise---known in the physics literature as the Kraichnan model~\cite{kraichnanSmallScaleStructure1968}. We thus give the first complete proof of a turbulent statistical scaling law like~\eqref{eq:kolmogorov 11 thirds} or~\eqref{eq:obukhov corrsin} in the ``rough'' regime---in which the fluid velocity has less than one derivative, corresponding to the limit of vanishing fluid viscosity. The ``smooth'' regime, also known as the Batchelor regime and corresponding to a fixed positive fluid viscosity, is by now well understood; see the discussion in Section~\ref{sss:batchelor} for an overview and connections with the problem under study.

\subsection{Results for two Kraichnan models}

We now specify two interesting velocity fields for which we have complete results leading to a version of~\eqref{eq:obukhov corrsin}. As will be clear in Section~\ref{ss:general}, we can prove some form of anomalous regularization for a broad variety of different transport noises, and we can show that anomalous dissipation and regularization imply a form of~\eqref{eq:obukhov corrsin} for a very general class of correlated-in-time velocity fields as well as transport noises. We provide these two specific cases first in order to give a clear statement of results free from the technical assumptions needed in the general setting.

In this section, we focus on the transport noise (or Kraichnan) setting; that is, we consider 
\begin{equation}
    \dot \theta_t^\kappa -\kappa \Delta \theta_t^\kappa + \circ du_t \cdot \nabla \theta_t^\kappa= F(x)dW_t,
    \label{eq:kraichnan-forced}
\end{equation}
where $\theta_t^\kappa : \T^d \to \R$, $F \in L^2(\T^d)$ with $\int F(x)\,dx =0$, $W_t \in \R$ a standard Brownian motion, and $du_t$ is a white-in-time, correlated-in-space, incompressible Gaussian vector field independent of $W_t$, defined by
\begin{equation}
    \label{eq:dut def}
    du_t(x) = \sum_{k \in \Z^d \backslash \{0\}} w_k e^{2\pi i k\cdot x}\sum_{j=1}^{d-1}\mathrm{e}_{k,j} dW_t^{k,j},
\end{equation}
where $\mathrm{e}_{k,j}$ is an orthonormal basis for $\{a \in \R^d : a \cdot k =0\}$, $w_k = w_{-k}$, $v_{-k,j} = v_{k,j}$, the $W^{k,j}$ are standard $\C$-valued Brownian motions, $W^{-k,j} = \overline{W^{k,j}}$, and $W^{k,j}$ independent from $W^{\ell,m}$ unless $k =  \pm \ell$ and $m=j$. Since $k \cdot \mathrm{e}_{k,j} =0$, we have that $\nabla \cdot du_t =0.$ The $\circ du_t$ in the above equations signifies that we are taking the Stratonovich convention for stochastic integration, which is the physically relevant choice in this setting (see Section~\ref{sss:kraichnan} for further discussion). We will also want to consider the freely-decaying version of~\eqref{eq:kraichnan-forced}:
\begin{equation}
    \label{eq:kraichnan-free}
\begin{cases}
    \dot \phi_t^\kappa -\kappa \Delta \phi_t^\kappa + \circ du_t \cdot \nabla \phi_t^\kappa=0,\\
    \phi_0^\kappa(x) = F(x).
\end{cases}
\end{equation}
Throughout the paper, we will always be in the case that for some $\alpha \in (0,1)$,
\begin{equation}
    \label{eq:dut-regularity}
    \sum_{k \in \Z^d \backslash \{0\}} |k|^{2\alpha} w_k^2 \leq 1.
\end{equation}
This assumption then gives that $du_t$ is (essentially) spatially $C^\alpha$. In this setting, since $\kappa>0$,~\eqref{eq:kraichnan-forced} and~\eqref{eq:kraichnan-free} are straightforwardly well defined (see, e.g., the discussion of~\cite[Section 2.2]{rowan_anomalous_2024}). 

We now specify the two choices of $w_k$ we will consider in this section. The first is (essentially) the usual choice of coefficients for the Kraichnan model on $\T^d$.

\begin{definition}[Isotropic Kraichnan model on $\T^d$]
\label{defn:Kraichnan standard}
       We let $d \geq 2$, fix $\alpha \in (0,1)$, and let $w_0 =0$. For $k \in \Z^d \backslash\{0\}$, we define 
    \[w_k := Z^{-1} |k|^{-d/2 - \alpha} (\log |k| +1)^{-1},\]
    where $Z$ is chosen so that~\eqref{eq:dut-regularity} holds with equality.
\end{definition}

We note that in the above definition we have an additional factor of $\log |k|$ compared to the more usual definitions~\cite{kraichnanSmallScaleStructure1968, falkovichParticlesFieldsFluid2001,galeati_anomalous_2024,rowan_anomalous_2024}. This is to make it so that~\eqref{eq:dut-regularity} actually holds (as opposed to being log divergent). This is essentially a choice made for purely technical convenience.  

We also define the following much sparser coefficient set.
\begin{definition}[Shear Kraichnan model on $\T^d$]
\label{defn:Kraichnan shear}
    We let $d \geq 2$, fix $\alpha \in (0,1)$, and let $w_0 =0$. For $k \in \Z^d \backslash\{0\}$, we define 
    \[w_k :=  \begin{cases}
        Z^{-1}|k|^{-1/2 - \alpha} (\log |k| +1)^{-1} & k_3 = k_4 =\cdots =k_d = 0,\\ 
        0 & \text{otherwise},
    \end{cases}\]
    where $Z$ is chosen so that~\eqref{eq:dut-regularity} holds with equality.
\end{definition}

We note that the velocity field given by Definition~\ref{defn:Kraichnan shear} is very anisotropic. It is a sum of a pure $x$ shear and a pure $y$ shear, so constant in every direction except $x$ and $y$, thus having a fairly degenerate structure. We emphasize that the velocity field $du_t$ with coefficients given by Definition~\ref{defn:Kraichnan shear} still \textit{points in all directions}---its range is $\R^d$---it just has a very restricted Fourier support: the $k_x$ and $k_y$ axes.

With the above specifications of the coefficients $w_k$, we are ready to state our first main result, giving anomalous dissipation and anomalous regularization for passive scalars advected by the ``Kraichnan models'' specified above.

\begin{theorem}[Anomalous dissipation and regularization]
\label{thm:anomalous dissipation and regularization special}
    Let $d\geq 2$, $\alpha \in (0,1)$, and let $du_t$ have its coefficients defined by either Definition~\ref{defn:Kraichnan standard} or Definition~\ref{defn:Kraichnan shear}. Then there exists $C(d,\alpha)>0$ such that for all $F \in L^2(\T^d)$ with $\int F(x)\,dx = 0$ and all $\kappa>0$, if we let $\phi^\kappa_t$ be the solution to~\eqref{eq:kraichnan-free}, then we have the following estimates:
    \begin{align}
        \text{for all } t \geq 0, \quad \E \|\phi^\kappa_t\|_{L^2(\T^d)}^2 &\leq C e^{-C^{-1}t} \|\phi^\kappa_0\|_{L^2(\T^d)}^2,
        \label{eq:anomalous dissipation special}\\
        \E \int_0^1 \sum_{k \in \Z^d \backslash \{0\}}  \frac{|k|^{2(1-\alpha)}}{(\log |k|+ 1)^{4}} |\hat \phi^\kappa_t(k)|^2\,dt  &\leq  C \|\phi^\kappa_0\|_{L^2(\T^d)}^2,
        \label{eq:anomalous regularization special}
    \end{align}
    where $\hat \phi^\kappa_t(k)$ is the $k$th Fourier coefficient of $\phi^\kappa_t$.
\end{theorem}

We emphasize that the constants in the above estimates are \textit{uniform in $\kappa>0$}. As such,~\eqref{eq:anomalous dissipation special} is a statement of anomalous dissipation:~\eqref{eq:kraichnan-free} is dissipating $L^2$ energy uniformly in $\kappa>0$ despite the $\kappa =0$ equation being formally $L^2$ conserving (see Section~\ref{sss:anomalous dissipation} for further discussion of this phenomenon).~\eqref{eq:anomalous regularization special} is a statement of anomalous regularization:~\eqref{eq:kraichnan-free} is gaining regularity (since $1-\alpha>0$) on its initial data uniformly in $\kappa>0$ despite the $\kappa=0$ equation being regularity preserving---at least for smooth velocity fields (see Section~\ref{sss:anomalous regularization} for further discussion of this phenomenon).

\begin{remark}
    In~\eqref{eq:anomalous regularization special}, we see we are getting (essentially) $H^{1-\alpha}_x$ regularity on $\phi^\kappa_t$ (uniformly in $\kappa>0$). This is however in conflict with~\eqref{eq:obukhov corrsin}, which predicts (essentially) $H^{\frac{1-\alpha}{2}}_x$ regularity in equilibrium. This discrepancy is due to~\eqref{eq:obukhov corrsin} only being the correct scaling prediction when the advecting velocity field is time-regular (say at least $C^0_{t,x}$). In the case we are considering, $du_t$ is white in time, so (essentially) $C^{-1/2}_t C^\alpha_x$. This roughness in time~\cite{gawedzki_anomalous_1995,falkovichParticlesFieldsFluid2001,galeati_anomalous_2024} changes the scaling of~\eqref{eq:obukhov corrsin} to 
    \begin{equation}
        \label{eq:obukhov corrsin kraichnan}
        \E |\hat \theta(k)|^2 \approx |k|^{-d - 2(1-\alpha)}.
    \end{equation}
    That is we expect (essentially) $H^{1-\alpha}_x$ regularity in equilibrium, in exact correspondence with~\eqref{eq:anomalous regularization special} (up to logarithmic corrections).
\end{remark}

\begin{remark}
    Our anomalous regularization statement~\eqref{eq:anomalous regularization special} is integrated in time. One might wonder if a pointwise-in-time anomalous regularization statement also holds. Such a pointwise-in-time statement is proven in~\cite[Theorem 1.3]{drivas_anomalous_2025}. Following the same argument of~\cite[Proposition 4.11]{drivas_anomalous_2025}, we too could upgrade~\eqref{eq:anomalous regularization special} to a pointwise-in-time statement but at the cost of strictly decreasing the regularity exponent to $1-\alpha - \delta$ for any $\delta>0$. It seems likely that one could get a pointwise-in-time version of~\eqref{eq:anomalous regularization special} at the endpoint regularity of $1-\alpha$ with additional log corrections by keeping careful quantitative track of the constants appearing in the proof of~\cite[Proposition 4.11]{drivas_anomalous_2025}. This would require  estimates on the constants appearing in (infinitely many) interpolation theorems, which likely would require reproving these theorems. Since the pointwise-in-time estimate on the free-decay problem doesn't gain us any better estimates on the invariant measure, which is the primary object of interest in this work, we leave such considerations to future studies.
\end{remark}

These anomalous dissipation and regularization estimates then give rise to the following version of the statistical scaling law~\eqref{eq:obukhov corrsin kraichnan}. 

\begin{theorem}[Obukhov--Corrsin spectrum on annuli]
\label{thm:OC special}
    Let $d\geq 2$, $\alpha \in (0,1)$, and let $du_t$ have its coefficients defined by either Definition~\ref{defn:Kraichnan standard} or Definition~\ref{defn:Kraichnan shear}. Then there exists $C(d,\alpha)>0$ such that for all $F \in L^2(\T^d)$ with $\int F(x)\,dx = 0$ and all $\kappa>0$, we have the following results. The Markov process $\theta^\kappa_t$ given by~\eqref{eq:kraichnan-forced} has a unique invariant probability measure $\mu^\kappa$ on the zero-mean subspace of $L^2(\T^d)$, and $\mu^\kappa$ has the global moment regularity upper bound
    \begin{equation}
    \label{eq:regularity for invariant measure special}
    \E_{\mu^\kappa} \sum_{k \in \Z^d \backslash\{0\}} \frac{|k|^{2(1-\alpha)}}{(\log |k| +1)^4} |\hat \theta^\kappa|^2 \leq C \|F\|_{L^2(\T^d)}^2,
    \end{equation}
    where $\hat \theta^\kappa(k)$ is the $k$th Fourier coefficient of $\theta^\kappa \in L^2(\T^d)$. We also have the following upper and lower bounds on the Fourier mass on annuli: there exists $r_0(F)>0$ such that for all $r \in (\kappa^{\frac{1}{2\alpha}}, r_0)$, 
     \begin{align*}C^{-1} r^{2(1-\alpha)} \|F\|_{L^2}^2 &\leq \sum_{ C^{-1} r^{-1}(\log r^{-1})^{-\frac{2}{\alpha}}  \leq |k| \leq C r^{-1} (\log r^{-1})^{\frac{2}{1-\alpha}}} \E_{\mu^\kappa} |\hat \theta^\kappa(k)|^2 
     \\&\qquad\qquad\qquad\qquad\qquad\qquad\qquad\leq  C (\log r^{-1})^{\frac{4}{\alpha}} r^{2(1-\alpha)} \|F\|_{L^2}^2.
     \end{align*}
\end{theorem}

We note that the final display of Theorem~\ref{thm:OC special} is precisely the predicted statistical scaling law~\eqref{eq:obukhov corrsin kraichnan} summed over a geometrically sized annulus on the Fourier lattice, $\{C^{-1} r^{-1} \leq |k| \leq C r^{-1}\}$, at the length scale $r$ (so wavenumber magnitude $r^{-1}$)---up to logarithmic corrections. There are multiple logarithmic corrections. First, the annulus is not quite of geometric size: it instead grows very slightly faster than geometrically. Then since the annulus is ``too large'', we expect a bit more mass on the annulus by summing~\eqref{eq:obukhov corrsin kraichnan}, so the lower bound is very slightly too small (by logarithmic factors). Finally the upper and lower bounds are mismatched by a logarithmic factor. These logarithmic corrections however are essentially physically unobservable; it is not even clear they should be able to be completely removed. Nonetheless, we do not claim that the logarithmic factors are optimal, and in particular we would expect a matching upper and lower bound to hold in at least some cases. There is also the (natural) restriction that $r > \kappa^{1/2\alpha}$, which is further discussed in Remark~\ref{rmk:dissipation range} below.

The fact that we have to sum the original pointwise prediction~\eqref{eq:obukhov corrsin kraichnan} over annuli in order to get the lower bound (a pointwise upper bound clearly holds as a consequence of~\eqref{eq:regularity for invariant measure special}) is not too surprising; a similar summing is needed to get bounds on the Batchelor spectrum (see Section~\ref{sss:batchelor} for further discussion and comparison). While a pointwise lower bound may hold, its proof would likely require much finer control on the Fourier evolution of~\eqref{eq:kraichnan-free}.

The final result we record before stating our results in their general formulation is the following infinite-order smoothing estimate in the case of the velocity field $du_t$ defined by Definition~\ref{defn:Kraichnan shear}. In that case, we note for $d \geq 3$ and any $3 \leq j \leq d$, we have that $\partial_j du_t =0$. Thus, fixing $3 \leq j \leq d$ and letting $|\partial_j| := (- \partial_j \partial_j)^{1/2}$, for any $s \in [0,\infty)$, if $\phi^\kappa_t$ solves~\eqref{eq:kraichnan-free}, so does $|\partial_j|^s \phi^\kappa_t$ (with different initial data). This allows us to iterate the smoothing estimate~\eqref{eq:anomalous regularization special} to prove the following.

\begin{corollary}[Infinite order smoothing in constant directions]
    \label{cor:infinite smoothing in the z direction}
        Let $d\geq 3$, $\alpha \in (0,1)$, and let $du_t$ have its coefficients defined by Definition~\ref{defn:Kraichnan shear}. Then for all $n \in \N$ there exists $C(d,n,\alpha)>0$ such that for all $F \in L^2(\T^d)$ with $\int F(x)\,dx = 0$ and all $\kappa>0$, if we let $\phi^\kappa_t$ be the solution to~\eqref{eq:kraichnan-free}, then for all $3 \leq j \leq d$, we have the following estimate
    \begin{equation}
        \label{eq:infinite order smoothing}
        \E \|\partial_j^n \phi^\kappa_1\|_{L^2(\T^d)}^2  \leq  C \|\phi^\kappa_0\|_{L^2(\T^d)}^2.
    \end{equation}
\end{corollary}

This result yields anomalous infinite order smoothing in some directions. This is somewhat surprising but not in conflict with other known theory, such as the Constantin--E--Titi-type commutator estimates~\cite{constantin_onsagers_1994}, since $\phi^\kappa_t$ will still be rough (only $C^{(1-\alpha)-}_x$) in the $x_1, x_2$ directions. This curious result is a consequence of having anomalous regularization estimates even for somewhat ``degenerate'' velocity fields---such as that given by Definition~\ref{defn:Kraichnan shear}---that have an a.s.\ exact symmetry (here being invariant under $x_j$ translations for $3 \leq j \leq d$). Since this corollary does not appear to be of significant physical interest and, as such, we do not discuss it further.

\subsection{Notation and general results}
\label{ss:general}

We now state our general results which will give Theorem~\ref{thm:anomalous dissipation and regularization special} and Theorem~\ref{thm:OC special} as essentially direct consequences (see Section~\ref{ss:theorems to theorems} for the relevant arguments). We first clarify and define some notation.

\subsubsection{Notation, definitions, and assumptions}

Throughout, we work on the torus $\T^d$ for $d \geq 2$, which we identify with $[0,1]^d /\sim$.  For $j \in \Z^d \backslash\{0\},$ we denote by $j^\perp \subseteq \R^d$ the orthogonal complement of the subspace of $\R^d$ spanned by $j$. Then $\Pi_{j^\perp}$ denotes orthogonal projection onto this subspace. For a vector $a  \in \Z^d \backslash \{0\}$, we let $\<a\>$ denote its span in $\Z^d$:
\[\<a\>:= \{r a : r \in \Z\}.\]

In addition to the stochastic equations~\eqref{eq:kraichnan-forced} and~\eqref{eq:kraichnan-free}, we will also consider the usual advection-diffusion equation, both with a stochastic forcing and in free decay:
\begin{align}
    \label{eq:advection-diffusion-forced}
    \dot \theta_t^\kappa -\kappa \Delta \theta_t^\kappa + v_t \cdot \nabla \theta_t^\kappa = F(x) dW_t,\\
     \label{eq:advection-diffusion-free}
\begin{cases}
    \dot \phi_t^\kappa -\kappa \Delta \phi_t^\kappa + v_t \cdot \nabla \phi_t^\kappa =0,\\
     \phi_0^\kappa(x) = F(x),
    \end{cases}
\end{align}
where $\kappa>0$ and as with $du_t$, we always take $v_t$ to be divergence-free: $\nabla \cdot v_t =0.$ $(v_t)_{t \in \R}$ will be taken to be random and independent of the forcing noise $W_t$. We assume throughout that $(v_t)_{t \in \R}$ is \textit{time-stationary} in law: for all $s \in \R$, $(v_{t + s})_{t \in \R} \stackrel{d}{=} (v_t)_{t \in \R}$. We will always assume that for some $\alpha \in (0,1),
$
    \begin{equation}
    \label{eq:v regularity}
    \sup_{t \in \R} \|v_t\|_{C^\alpha_x} \leq 1 \quad \text{and} \quad v \in C^0(\R \times \T^d).
    \end{equation}
We note then that~\eqref{eq:advection-diffusion-forced} and~\eqref{eq:advection-diffusion-free} are well-posed as $\kappa>0$, even though the $\kappa=0$ equations may not be well-posed.

Throughout we will assume that
\[\int_{\T^d} \theta^\kappa_t(x)\,dx = \int_{\T^d} \phi^\kappa_t(x)\,dx = \int_{\T^d} du_t(x)\,dx = \int_{\T^d} v_t(x)\,dx =0.\]
Since the advecting flows are divergence-free, this condition will hold for $\theta^\kappa_t$ and $\phi^\kappa_t$ provided $\int F(x)\,dx =0$. Since averages are conserved by the advection-diffusion equation, anomalous dissipation estimates such as~\eqref{eq:anomalous dissipation special} cannot hold (without modification) unless we assume that $\int F(x)\,dx =0$. We thus take this zero-mean condition as a global assumption for notational simplicity. Similarly, averages in the advecting flows induce trivial transformations on the solutions that we are not interested in, so we also take them to be zero-mean. 

For a function $f : \T^d \to \R$, we denote the Fourier transform $\hat f : \Z^d \to \C$ so that
\[f(x) = \sum_{k \in \Z^d} \hat f(k) e^{2\pi i k \cdot x}.\]
Because we assume all of our functions to be zero-mean throughout, we variously view the Fourier transform as a function $\Z^d \to \C$ and a function $\Z^d \backslash\{0\} \to \C$.

We define $\sigma(H)$ regularity spaces which generalize $H^s$ spaces.
\begin{definition}[$\sigma(H)$-spaces]
    For a function $\sigma : [1,\infty) \to (0,\infty)$, we define the $\sigma(H)$ norm on zero-mean functions,
    \[\|f\|_{\sigma(H)} := \Big(\sum_{k \in \Z^d \backslash\{0\}} \sigma(|k|)^2 |\hat f(k)|^2 \Big)^{1/2}.\]
\end{definition}
We note that $\sigma(H) = H^s$ for $\sigma(r) = r^s$. We now give a general definition of what we mean for the velocity fields $du_t$ or $v_t$ to exhibit anomalous dissipation or regularization. We emphasize that here and throughout the paper, all constants are \textit{independent of} $\kappa>0.$

\begin{definition}[Anomalous regularization and dissipation]
\label{defn:anomalous}
    We say that the velocity field $du_t$ \textit{exhibits anomalous dissipation} if there exists $C>0$ such that for all $F \in L^2(\T^d)$ with $\int F(x)\,dx =0$ and all $\kappa \in (0,1]$, letting $\phi^\kappa_t$ be the solution to~\eqref{eq:kraichnan-free}, then for all $t>0,$
    \[\E \|\phi^\kappa_t\|_{L^2(\T^d)}^2 \leq C e^{- C^{-1} t} \|\phi_0\|_{L^2(\T^d)}^2.\]
    For a function $\sigma : [1,\infty) \to (0,\infty)$, we say that $du_t$ \textit{exhibits anomalous regularization up to $\sigma(H)$} if there exists $C>0$ such that for all $F \in L^2(\T^d)$ with $\int F(x)\,dx =0$ and all $\kappa \in (0,1]$, letting $\phi^\kappa_t$ be the solution to~\eqref{eq:kraichnan-free}, then
\[
        \E \int_0^1 \|\phi^\kappa_t\|_{\sigma(H)}^2\,dt \leq C\|\phi_0\|_{L^2(\T^d)}^2.
\]
    We take the analogous definitions for a time-correlated velocity field $v_t$, with~\eqref{eq:advection-diffusion-free} in place of~\eqref{eq:kraichnan-free}.
\end{definition}

\subsubsection{Bounds on the energy spectrum of the invariant measure}

We next want to state bounds on the invariant measures associated to~\eqref{eq:kraichnan-forced} and~\eqref{eq:advection-diffusion-forced}. As such, we first need a result guaranteeing these objects exist uniquely. This is rather straightforward as we always have $\kappa>0$, giving a global geometric contractivity under the natural coupling of different initial conditions (giving them the same noise). There is however one minor complication, which is that while we can consider $\theta^\kappa_t$ to be the Markov process on its own in the case of a white-in-time noise~\eqref{eq:kraichnan-forced}, it is no longer (necessarily) a Markov process for~\eqref{eq:advection-diffusion-forced}. For~\eqref{eq:advection-diffusion-forced}, we must instead keep track of the entire velocity field trajectory in order to ensure the Markov condition. That is the Markov process we consider is $((v_{s+t})_{s \in \R}, \theta^\kappa_t) \in C^0(\R \times \T^d) \times L^2(\T^d)$. It is then straightforward to verify this is indeed a Markov process.

\begin{proposition}[Invariant measures]
\label{prop:invariant measures}
    For any $\kappa>0$, $F \in L^2(\T^d)$ such that $\int F(x)\,dx =0$, and $w_k$ satisfying~\eqref{eq:dut-regularity} for some $\alpha\in(0,1)$, the Markov process $\theta^\kappa_t$ given by~\eqref{eq:kraichnan-forced} has a unique invariant probability measure $\mu^\kappa$ on the subspace of $L^2(\T^d)$ given by zero-mean functions. Further, letting $\phi^\kappa_t$ be the solution to~\eqref{eq:kraichnan-free}, then for all $k \in \Z^d \backslash \{0\},$ we have that
    \begin{equation}
        \label{eq:equilibrium equals free decay}
        \E_{\mu^\kappa} |\hat \theta^\kappa(k)|^2 = \int_0^\infty \E |\hat \phi^\kappa_t(k)|^2\,dt.
    \end{equation}

    Similarly, for any $\kappa>0$, $F \in L^2(\T^d)$ with $\int F(x)\,dx =0$, and $(v_t)_{t \in \R}$ having time-stationary law $\nu$ on $C^0(\R \times \T^d)$ almost surely satisfying~\eqref{eq:v regularity} for some $\alpha \in (0,1)$, the Markov process $((v_{s+t})_{s \in \R}, \theta^\kappa_t)$ where $\theta^\kappa_t$ solves~\eqref{eq:advection-diffusion-forced} has a unique invariant probability measure $\mu^\kappa$ on the subspace of $C^0(\R \times \T^d) \times L^2(\T^d)$ given by $((v_s)_{s \in \R}, \theta^\kappa)$ where $\int \theta^\kappa(x)\,dx =0$ such that $\mu^\kappa(d(v_s)_{s \in \R}, L^2(\T^d)) = \nu(d(v_s)_{s \in \R})$.\footnote{That is, for $((v_s)_{s \in \R}, \theta^\kappa)$ distributed according to $\mu^\kappa$, $v$ has the marginal law given by $\nu$.} Further, letting $\phi^\kappa_t$ be the solution to~\eqref{eq:advection-diffusion-free}, for all $k \in \Z^d \backslash \{0\},$ we have that
    \[\E_{\mu^\kappa} |\hat \theta^\kappa(k)|^2 = \int_0^\infty \E |\hat \phi^\kappa_t(k)|^2\,dt.\]
\end{proposition}

We now state the straightforward proposition that, if we have anomalous dissipation and anomalous regularization up to some regularity, then we get moment bounds in that regularity for the invariant measure.

\begin{proposition}[Upper bounds on the invariant measure]
    \label{prop:OC upper bounds}
    Suppose that $du_t$ has coefficients $w_k$ satisfying~\eqref{eq:dut-regularity} for some $\alpha \in (0,1)$ and exhibits anomalous dissipation and anomalous regularization up to $\sigma(H)$ for some $\sigma : [1,\infty) \to (0,\infty).$ Fix $F \in L^2(\T^d)$ with $\int F(x)\,dx =0$ and $\kappa \in (0,1]$, then let $\mu^\kappa$ be unique invariant measure for~\eqref{eq:kraichnan-forced} from Proposition~\ref{prop:invariant measures}. Then there exists a $C>0$, depending on the constants in Definition~\ref{defn:anomalous}, such that
    \[\E_{\mu^\kappa} \|\theta^\kappa\|_{\sigma(H)}^2 \leq C \|F\|_{L^2(\T^d)}^2.\]
    The same result holds with $v_t$ satisfying~\eqref{eq:v regularity} for some $\alpha \in (0,1)$ in place of $du_t$ and~\eqref{eq:advection-diffusion-forced} in place of~\eqref{eq:kraichnan-forced}.
\end{proposition}

We now state our general result for correlated-in-time advecting flows giving bounds on the energy spectrum of the invariant measure under the assumption of anomalous dissipation and regularization up to some general regularity index.

\begin{theorem}[Lower bounds on the invariant measure for correlated-in-time models]
\label{thm:correlated lower bound}

    Let $v_t$ almost surely satisfy~\eqref{eq:v regularity} for some $\alpha \in (0,1)$. Suppose that $v_t$ exhibits anomalous dissipation and anomalous regularization up to $\sigma(H)$ for some $\sigma : [1,\infty) \to (0,\infty)$ with
    \begin{equation}
        \label{eq:sigma lower bound correlated}
           \sigma(r) \geq \frac{r^\beta}{(\log r + 1)^m}.
    \end{equation}
 
    Then for any $F \in L^2(\T^d)$ with $\int F(x)\,dx =0$, there exists $r_0(F) >0$ and a constant $C>0$, depending on $\beta, m,$ and the constants appearing in the anomalous dissipation and regularization estimates for $v_t$, such that for all $r \in (\kappa^{\frac{1}{1+\alpha}}, r_0),$
    \begin{equation}
    \label{eq:lower bound correlated general}
    \sum_{C^{-1} r^{- \frac{1+\alpha}{2(1-\beta)}}(\log r^{-1})^{-\frac{m}{1-\beta}} \leq |k| \leq C r^{-\frac{1-\alpha}{2\beta}} (\log r^{-1})^{\frac{m}{\beta}}} \E_{\mu^\kappa} |\hat \theta^\kappa(k)|^2  \geq C^{-1} r^{1-\alpha} \|F\|_{L^2(\T^d)}^2,
    \end{equation}
    where $\mu^\kappa$ is the unique invariant measure from Proposition~\ref{prop:invariant measures}.
    In particular when $\beta = \frac{1-\alpha}{2}$, for all $r \in (\kappa^{\frac{1}{1+\alpha}}, r_0),$ we have
        \begin{align}
     C^{-1} r^{1-\alpha} \|F\|_{L^2(\T^d)}^2 &\leq \sum_{C^{-1} r^{- 1}(\log r^{-1})^{-\frac{2m}{\alpha+1}} \leq |k| \leq C r^{-1} (\log r^{-1})^{\frac{2m}{1-\alpha}}} \E_{\mu^\kappa} |\hat \theta^\kappa(k)|^2 
     \notag\\& \qquad\qquad\qquad\qquad\qquad\leq C (\log r^{-1})^{\frac{4m}{1+\alpha}} r^{1-\alpha} \|F\|_{L^2(\T^d)}^2. 
        \label{eq:lower bound correlated particular}
    \end{align}
\end{theorem}

Next we give our general result for white-in-time advecting flows giving bounds on the energy spectrum of the invariant measure. It is this result that will imply Theorem~\ref{thm:OC special}.

\begin{theorem}[Lower bounds on the invariant measure for white-in-time models]
\label{thm:white lower bound}
    Let $du_t$ have coefficients $w_k$ satisfying~\eqref{eq:dut-regularity} for some $\alpha \in (0,1)$. Suppose that $du_t$ exhibits anomalous dissipation and anomalous regularization up to $\sigma(H)$ for some $\sigma : [1,\infty) \to (0,\infty)$ with
    \begin{equation}
        \label{eq:sigma lower bound kraichnan}
           \sigma(r) \geq \frac{r^\beta}{(\log r + 1)^m}.
    \end{equation}
 
    Then for any $F \in L^2(\T^d)$ with $\int F(x)\,dx =0$, there exists $r_0(F) >0$ and a constant $C>0$, depending on $\beta, m,$ and the constants appearing in the anomalous dissipation and regularization estimates for $du_t$, such that for all $r \in (\kappa^{\frac{1}{2\alpha}}, r_0),$
    \begin{equation}
        \label{eq:lower bound white general}
          \sum_{ C^{-1} r^{-\frac{\alpha}{1-\beta}}(\log r^{-1})^{-\frac{m}{1-\beta}}  \leq |k| \leq C r^{-\frac{1-\alpha}{\beta}} (\log r^{-1})^{\frac{m}{\beta}}} \E_{\mu^\kappa} |\hat \theta^\kappa(k)|^2  \geq C^{-1} r^{2(1-\alpha)} \|F\|_{L^2(\T^d)}^2,
    \end{equation}
    where $\mu^\kappa$ is the unique invariant measure from Proposition~\ref{prop:invariant measures}.
    In particular when $\beta = 1-\alpha$, for all $r \in (\kappa^{\frac{1}{2\alpha}}, r_0),$ we have
        \begin{align}
     C^{-1} r^{2(1-\alpha)} \|F\|_{L^2(\T^d)}^2 &\leq \sum_{ C^{-1} r^{-1}(\log r^{-1})^{-\frac{m}{\alpha}}  \leq |k| \leq C r^{-1} (\log r^{-1})^{\frac{m}{1-\alpha}}} \E_{\mu^\kappa} |\hat \theta^\kappa(k)|^2  
     \notag\\&
     \qquad\qquad\qquad\qquad
    \leq  C (\log r^{-1})^{\frac{2m}{\alpha}} r^{2(1-\alpha)} \|F\|_{L^2(\T^d)}^2.
         \label{eq:lower bound white particular}
    \end{align}
\end{theorem}

\begin{remark}
\label{rmk:dissipation range}
    The above results, as well as Theorem~\ref{thm:OC special}, include the restriction that $r > \kappa^{\frac{1}{2\alpha}}$ in the white-in-time case and $r> \kappa^{\frac{1}{1+\alpha}}$ in the time correlated case. It is clear there must be some $\kappa$ dependent lower bound on $r$, since the presence of the diffusion $\kappa \Delta$ ensures that
    \[\E_{\mu^\kappa} \|\theta^\kappa\|_{H^1}^2 \leq \kappa^{-1} \|F\|_{L^2(\T^d)}^2.\]
    Thus on very small scales (dependent on $\kappa$), we must have the Fourier coefficients of $\theta^\kappa$ must decay faster than~\eqref{eq:obukhov corrsin} or~\eqref{eq:obukhov corrsin kraichnan} predict. The interval on which the Fourier coefficients decay more quickly due to the dissipation is known as the \textit{dissipation range}. The values we get for the transition to the dissipation range are in agreement with the values predicted by dimensional analysis, as is suggested by the agreement of $\frac{1}{1+\alpha} = \frac{3}{4}$ when $\alpha = 1/3$, agreeing with the Kolmogorov lengthscale of K41 theory.
\end{remark}

The above results are of substantially less interest when $\beta < \frac{1-\alpha}{2}$ and $\beta < 1- \alpha$ for Theorem~\ref{thm:correlated lower bound} and Theorem~\ref{thm:white lower bound} respectively, since in that case we are really making algebraic-in-$r$ errors compared to summing the predictions~\eqref{eq:obukhov corrsin} and~\eqref{eq:obukhov corrsin kraichnan} respectively. However, they still provide nontrivial lower bounds that become increasingly good as $\beta$ increases to $\frac{1-\alpha}{2}$ or $1-\alpha$ respectively. In the $\beta = \frac{1-\alpha}{2}$ (or $\beta = 1-\alpha$) case we get matching upper and lower bounds on geometrically sized annuli as predicted by~\eqref{eq:obukhov corrsin}  (or~\eqref{eq:obukhov corrsin kraichnan})---where each of those claims is true up to logarithmic errors, just as in Theorem~\ref{thm:OC special}. These results thus show that (a version of) the Obukhov--Corrsin spectrum must hold whenever the advecting flow exhibits anomalous dissipation and anomalous regularization up to the sharp regularity space, modulo to logarithmic corrections. We note that while in Theorem~\ref{thm:anomalous dissipation and regularization special} and more generally in Theorem~\ref{thm:kraichnan}, we exhibit stochastic velocity fields $du_t$ that have the requisite anomalous dissipation and regularization properties, for the correlated-in-time case (to the best of our knowledge) proving that some velocity field $v_t$ satisfies the assumptions of Theorem~\ref{thm:correlated lower bound} for some $\beta >0$ remains an interesting open problem; see Section~\ref{ss:discussion} for further discussion.

\subsubsection{Anomalous regularization and dissipation for white-in-time models}

We now turn our attention to stating general conditions on the coefficients $w_k$ defining the stochastic advecting flow $du_t$ in~\eqref{eq:dut def} that give anomalous dissipation and regularization. It is by verifying these conditions (and bounding the function $S$ defined in~\eqref{eq:S-def}) for the special cases of Definition~\ref{defn:Kraichnan standard} and Definition~\ref{defn:Kraichnan shear} that we will get Theorem~\ref{thm:anomalous dissipation and regularization special}. Let us now give our hypotheses on the coefficients.

\begin{assumption}
    \label{asmp:wk}
    We fix $\alpha \in (0,1)$. Then we suppose that $(w_k)_{k \in \Z^d}$ is
    \begin{equation}
    \label{eq:C alpha condition}
    \sum_{k \in \Z^d \backslash\{0\}} |k|^{2\alpha} w_k^2 \leq 1.
\end{equation}
    We then let
\begin{equation}
\label{eq:S-def}
S(r) :=\sum_{|k| \leq r} |k|^{1+\alpha} w_k^2 .
\end{equation}
    We suppose the following structural conditions on $S$: there exists $r_0 \geq 4$, $\delta, \beta \in (0,1)$, and for all $K \geq 1$, $\Psi(K) >0$ such that for all $r \geq r_0$, we have the bounds
    \begin{align}
        \label{eq:S not too small} 
        S(r) &\geq \delta |r|^{\beta},
        \\
        \label{eq:nondegenerate}
    \inf_{v \in \R^d, |v|=1}\sum_{|k| \leq r}  |k|^{1+\alpha} w_k^2 |\Pi_{k^\perp} v| &\geq \delta S(r),\\
    \label{eq:geometric fluctuation bound}
S(Kr) &\leq \Psi(K) S(r).
    \end{align}
\end{assumption}

The first condition~\eqref{eq:C alpha condition} is (essentially) asking that $du_t \in C^\alpha_x$. We then define $S$ in~\eqref{eq:S-def}, which we will want to be as large as possible to get the maximal amount of regularization in~\eqref{eq:anomalous regularization kraichnan} below. As such, we are motivated to take $\alpha$ to be as large as possible so that~\eqref{eq:C alpha condition} holds (of course the relevant condition is that the sum is finite; if the sum is finite, we can rescale time in order to make that constant $1$). 

The final conditions~\eqref{eq:S not too small},~\eqref{eq:nondegenerate}, and~\eqref{eq:geometric fluctuation bound} are of a rather more technical nature.~\eqref{eq:S not too small} is almost necessary in order to get the anomalous dissipation. The condition is essentially requiring that $du_t$ doesn't have a full derivative, since if $du_t$ had a full spatial derivative, we'd have that for all $r>0$, $S(r) \leq \sum_{k} |k|^{1+\alpha} w_k^2 < \infty$, and so we couldn't possibly have the growth~\eqref{eq:S not too small} as $r \to \infty.$ A condition of this form is clearly necessary, as we couldn't possible have an anomalous dissipation estimate like~\eqref{eq:exponential decay} for spatially smooth advecting flows. The condition~\eqref{eq:geometric fluctuation bound} is purely technical and likely could be loosened in many different ways. However, it is a rather soft condition that is straightforwardly satisfied by all examples we are interested in.

The condition~\eqref{eq:nondegenerate} is more natural (and important) than~\eqref{eq:geometric fluctuation bound}. This condition essentially requires that all of the $w_k$ mass isn't concentrated on a single line in Fourier space. A condition like this is necessary, as the ``pure shear'' case---where e.g.\ $w_k = |k|^{-1/2 - \alpha} \indc_{k_2 = k_3 = \cdots = k_n  = 0}$ so that all of the mass of $w_k$ is purely concentrated on a line---does not exhibit anomalous dissipation or regularization (as can be shown by computing the solution explicitly). The condition~\eqref{eq:nondegenerate} quantitative non-degeneracy condition that excludes this case.

We now state our general result on anomalous dissipation and regularization.

\begin{theorem}[Anomalous regularization and dissipation for suitable white-in-time models]
\label{thm:kraichnan}
 Suppose $(w_k)_{k \in \Z^d \backslash\{0\}}$ satisfies Assumption~\ref{asmp:wk} for some $\alpha \in (0,1)$. Fix $F \in L^2(\T^d)$ with $\int F(x)\,dx =0$ and let $\phi^\kappa_t$ solve~\eqref{eq:kraichnan-free}. Then for all $R \geq r_0,$ there exists $C(\beta, \delta,R,\Psi)>0$ such that for all $\kappa \in (0,1],$ we have the bounds for all $t \geq 0$,
 \begin{align}
     \label{eq:exponential decay kraichnan}
     \E \|\phi^\kappa_t\|_{L^2(\T^d)}^2 \leq e^{-C^{-1}t}   \|\phi^\kappa_0\|_{L^2(\T^d)}^2,
     \\
     \label{eq:anomalous regularization kraichnan}
     \E \int_0^1 \sum_{k \in \Z^d \backslash \{0\}} \Big(\frac{S(R|k|)}{\log |k| + 1}\Big)^2 |\hat \phi^\kappa_t(k)|^2 \leq C \|\phi^\kappa_0\|_{L^2(\T^d)}^2,
 \end{align}
 where $S$ is given by~\eqref{eq:S-def}. Therefore, under Definition~\ref{defn:anomalous}, $du_t$ exhibits anomalous dissipation and anomalous regularization up to $\sigma_R(H)$ with
    \[\sigma_R(r) := \frac{S(Rr)}{\log r + 1}.\]
\end{theorem}

\subsection*{Acknowledgments}

We would like to thank Lucio Galeati, Sotirios Kotitsas, and Mario Maurelli for a stimulating discussion, as well Elias Hess-Childs for helpful feedback on an early draft.

\section{Discussion and previous results}

\label{ss:discussion}

We now provide further discussion of the results and give additional connections with the extensive mathematical literature.

\subsection{Anomalous dissipation, the Obukhov--Corrsin spectrum, and the importance of anomalous regularization}
\label{sss:anomalous dissipation}

The mathematical problem of anomalous dissipation of passive scalars---in which the solution to the advection-diffusion equation dissipates $L^2$ energy uniform in diffusivity despite the zero-diffusivity equation being formally energy-conserving---has seen substantial attention in recent years. While the results of the current work do include statements of anomalous dissipation in~\eqref{eq:anomalous dissipation special} and~\eqref{eq:exponential decay kraichnan}, they are not the primary interest of the paper.~\eqref{eq:exponential decay kraichnan} is a proper generalization of the previously known results, but the two most interesting cases covered in~\eqref{eq:anomalous dissipation special} were already treated in~\cite{rowan_anomalous_2024}. 

Anomalous dissipation is an important physical phenomenon underpinning the equilibrium (and non-equilibrium) theory of passive scalar turbulence and is a central ingredient to our arguments. Despite having a long history in the physics literature, appearing prominently in both the K41 theory of fluid turbulence and the Obukhov--Corrsin theory of passive scalar turbulence, the rigorous mathematical treatment of this phenomenon is quite a bit more recent. The first proof of anomalous dissipation for an advection-diffusion equation appeared in~\cite{drivas_anomalous_2022} and following that there have been many additional constructions, each focusing on a different aspect of the problem: see~\cite{colombo_anomalous_2023,armstrong_anomalous_2025,burczak_anomalous_2023,elgindi_norm_2024,hofmanova_anomalous_2025,johansson_anomalous_2024, hess-childs_universal_2025} among others.

Some of these results,~\cite{colombo_anomalous_2023,elgindi_norm_2024}, prove Obukhov--Corrsin-type upper bounds, proving that the scalar solution has uniform bounds in $C^\beta_x$ for all $\beta < \frac{1-\alpha}{2}$ for an advecting velocity field $u \in L^\infty_t C^\alpha_x$. Additionally, a qualitative form of the Obukhov--Corrsin lower bound~\cite[Section 5]{drivas_anomalous_2022} based on the argument of~\cite{constantin_onsagers_1994} is well known: if $\phi^\kappa$ exhibits anomalous dissipation, it cannot be uniformly bounded in $L^2_t C^\beta_x$ for any $\beta > \frac{1-\alpha}{2}$. 

However, these Obukhov--Corrsin-type bounds are of a rather different form than the ones considered here in Theorem~\ref{thm:OC special}, Theorem~\ref{thm:correlated lower bound}, and Theorem~\ref{thm:white lower bound}. The most obvious refinement is obtaining essentially sharp bounds on (almost) geometric annuli in Fourier space, providing a more precise localization of the Fourier mass, but there are also other substantial differences. We are working with the equilibrium measure for the stochastically forced equation~\eqref{eq:kraichnan-forced} while the previous references are considering the free-decay problem~\eqref{eq:advection-diffusion-forced}. It is in this distinction that the importance of anomalous regularization becomes clear. In the free-decay problem, one starts with smooth initial data and seeks to show that the solution to~\eqref{eq:advection-diffusion-free} doesn't become too singular in finite time. However, this proof technique straightforwardly fails in trying to prove the regularity of the equilibrium measure. From~\eqref{eq:equilibrium equals free decay}, in order to have positive regularity of the equilibrium measure, we need not just that the free decay problem doesn't become too large in a positive regularity norm, we actually need positive regularity norms to decay in a time integrable way, uniformly in diffusivity.

The only way we know how to get uniform-in-diffusivity decay of the solution is through anomalous dissipation, which gives decay in a zero-regularity norm. We then can use anomalous regularization, which bounds positive regularity norms by zero-regularity norms, to upgrade this decay to decay of positive regularity norms, giving the desired bound of a positive regularity norm in equilibrium. This suggests that to get Obukhov--Corrsin-type bounds in equilibrium, one typically needs the stronger phenomenon of anomalous regularization as opposed to the free-decay case, for which one only needs that the solutions don't become too rough.

We note additionally that the previously constructed deterministic examples exhibiting anomalous dissipation that have Obukhov--Corrsin-type regularity bounds are not ``time-uniform''. They all require one to start the advection-diffusion equation at specific times in order to have the correct behavior; this prevents the constructions from fitting into the framework considered here.

\subsection{The Kraichnan model/transport noise}
\label{sss:kraichnan}

The anomalous dissipation examples discussed above were all for the advection-diffusion equation~\eqref{eq:advection-diffusion-free}, in which the velocity field is at least bounded in time. The case of a stochastic advecting flow~\eqref{eq:kraichnan-free} also has a long history, originating in the physics literature~\cite{kraichnanSmallScaleStructure1968} as the ``Kraichnan model'' of a turbulent velocity field. Following its initial introduction, it has been widely utilized as a testing ground for ideas such as intermittency~\cite{gawedzki_anomalous_1995} and spontaneous stochasticity~\cite{bernard_slow_1998}. The model also saw mathematical treatment in~\cite{jan_integration_2002,jan_flows_2004,lototskii_passive_2004,lototsky_wiener_2006}.

We note that the presence of a multiplicative white (in time) noise in the equation~\eqref{eq:kraichnan-free} necessitates the choice of a stochastic integration convention, the two most common being It\^o and Stratonovich. In this setting though, the Stratonovich convention is certainly the right choice. The transport noise model should be thought of as the limiting case of a very fast fluctuating random velocity field. By the Wong-Zakai theorem~\cite{wongConvergenceOrdinaryIntegrals1965,wongRiemannStieltjesApproximationsStochastic1969}, this limiting procedure gives rise to a Stratonovich noise. See also~\cite{flandoli_additive_2022} for another perspective on the appearance of Stratonovich transport noise in fluid models.

More recently, there has been renewed mathematical interest in rough transport noise and the Kraichnan model~\cite{galeati_convergence_2020,flandoli_delayed_2021,flandoli_high_2021,coghi_existence_2024,rowan_anomalous_2024,galeati_anomalous_2024,drivas_anomalous_2025}. Of particular relevance for us here are~\cite{galeati_anomalous_2024,drivas_anomalous_2025}, but let us defer further discussion of these works to Section~\ref{sss:anomalous regularization}. In the author's previous work~\cite{rowan_anomalous_2024}, anomalous dissipation was proven for a wide variety of Kraichnan noises using techniques related to those of the current work (for further discussion, see Section~\ref{s:overview}).

\subsection{Anomalous regularization}
\label{sss:anomalous regularization}

Anomalous regularization, in which the passive scalar gains regularity over the initial data uniformly in diffusivity despite the zero-diffusivity equation being regularity preserving for spatially smooth flows (and generically causing growth of positive regularity norms), is a phenomenon that has only just started to be understood. While it is possible for $C^\alpha_x$ velocity fields to cause initial data to lose all Sobolev regularity~\cite{alberti_loss_2019}, anomalous regularization demonstrates that for certain fluid-like velocity fields, we actually see the opposite phenomenon: substantial regularity gain. The connection between anomalous regularization and the appearance of the statistical scaling laws like~\eqref{eq:kolmogorov 11 thirds} and~\eqref{eq:obukhov corrsin} has previously been discussed in~\cite{drivas_self-regularization_2022}.

Mathematical demonstrations of anomalous regularization are rather recent: following~\cite{coghi_existence_2024}, which used an anomalous regularization estimate to prove well-posedness for the 2D vorticity-form Euler equations with rough transport noise, anomalous regularization has been demonstrated for a variety of equations subject to transport noise~\cite{jiao_well-posedness_2025,bagnara_anomalous_2024,bagnara_regularization_2025}. Of particular relevance to the current work are~\cite{galeati_anomalous_2024} and~\cite{drivas_anomalous_2025}, which consider the same equation~\eqref{eq:kraichnan-free} studied here, but on $\R^d$ instead of $\T^d.$ We note that~\cite{drivas_anomalous_2025} also considers a variety of interesting cases where $\nabla \cdot du_t \ne 0$, which we do not cover at all.

The arguments of~\cite{galeati_anomalous_2024} and~\cite{drivas_anomalous_2025} do not directly transfer from $\R^d$ to $\T^d$. The behavior of equilibrium measures is much better on $\T^d$ due to its compactness and, as such, for estimates like the Obukhov--Corrsin spectrum we prefer to work on $\T^d$. The proof techniques of~\cite{galeati_anomalous_2024} and~\cite{drivas_anomalous_2025} exploit the rotational symmetry of $\R^d$ and use exact asymptotic expansions in order to prove the anomalous regularization. Since there is no (global) rotational symmetry on the torus, this proof method runs into difficulties on $\T^d$. We also note that due to the exact asymptotic expansions present in their arguments, they cannot handle the case where $du_t = du^1_t + du^2_t$ for independent noises where $du^1_t$ is well-behaved (satisfying their assumptions) and $du^2_t$ is some arbitrary noise (with enough regularity so that the equation is still well-posed). In contrast, due to the ``energy estimate'' structure of our proof, we can treat a very broad class of highly anisotropic noises and more noise is always ``helpful'': adding an additional negative term that could always be disregarded. This allows us to cover the case that $du_t = du^1_t + du^2_t$, as can be seen by inspecting the proof (we choose not to make statements along these lines so as not to further complicate the already quite technical theorem statements).

We note that~\cite[Theorem 1.3, (1.11)]{drivas_anomalous_2025} provides a sharp endpoint regularity bound in a Besov space with regularity exactly $1-\alpha$, strictly better than our result~\eqref{eq:anomalous regularization special} which is off the endpoint by some logarithmic correction. However, one can verify that even if we were to attain an anomalous regularization result in the same Besov space, there would still be logarithmic corrections present in the Obukhov--Corrsin bounds of Theorem~\ref{thm:OC special}. In this work, we generally won't be concerned with optimality of logarithmic errors; the sharp endpoint regularization for the Kraichnan model on the torus is left for future studies.

Finally, we note the recent work~\cite{hess-childs_turbulent_2025} gives an anomalous regularization result for a deterministic (and positive time regularity) flow. However, this result is both non-sharp---not getting close to the endpoint regularity of $\frac{1-\alpha}{2}$ for the scalar---and non-time-uniform---only giving anomalous regularization when started from specific times. As such, this construction doesn't satisfy the hypotheses of Theorem~\ref{thm:correlated lower bound}.  

\subsection{Batchelor regime passive scalars}
\label{sss:batchelor}

We now divert our discussion to a different regime of the advection-diffusion equation: the Batchelor regime. In the Batchelor regime, we take the advecting flow to be spatially smooth, heuristically corresponding to a fluid at a fixed positive viscosity. In the Batchelor regime, there is a different prediction for the scaling behavior of passive scalars at statistical equilibrium analogous to~\eqref{eq:obukhov corrsin}, known as Batchelor's law:
\begin{equation}
\label{eq:batchelor}
\E_\mu |\hat \theta(k)|^2 \approx |k|^{-d}.
\end{equation}
We note this formally corresponds to~\eqref{eq:obukhov corrsin} (or~\eqref{eq:obukhov corrsin kraichnan}) with $\alpha =1$, but this is the correct prediction even for velocity fields with $C^\infty_x$ regularity; the statistics of the passive scalar become independent of the regularity of the advecting flow for regularities above one. We note also that this prediction is now independent of whether one considers the advection-diffusion case~\eqref{eq:advection-diffusion-forced} or the transport noise case~\eqref{eq:kraichnan-forced}. 

In the Batchelor regime, the primary phenomena of interest (in place of anomalous dissipation and anomalous regularization) are (exponential) mixing~\cite{alberti_exponential_2019,elgindi_universal_2019,bedrossian_almost-sure_2022,myers_hill_exponential_2022,blumenthal_exponential_2023, luo_elementary_2024, navarro-fernandez_exponential_2025,cooperman_exponential_2025} and enhanced dissipation~\cite{feng_dissipation_2019,zelati_relation_2020,bedrossian_almost-sure_2021,cooperman_harris_2025,elgindi_optimal_2025}. In \cite{bedrossian_batchelor_2021}, it is shown how the competition of exponential mixing (which sends Fourier mass to infinity at an exponential rate) and the regularity of the advecting flow (which prevents Fourier mass from going to infinity faster than exponentially) necessarily leads to a cumulative version of the Batchelor spectrum~\eqref{eq:batchelor}, given by summing over balls in Fourier space:
\[\E \sum_{|k| \leq r} |\hat \theta(k)|^2 \approx \log r.\]
In~\cite{cooperman_fourier_2025}, a refined argument gives upper and lower bounds on annuli of constant width. The argument of this work giving Obukhov--Corrsin bounds on annuli in Fourier space draws inspiration from~\cite{cooperman_fourier_2025}, as is further discussed in Section~\ref{s:overview}.

\subsection{Uniform-in-diffusivity vs.\ at zero diffusivity}

The final topic we discuss is the issue of uniform-in-diffusivity estimates compared to estimates for the zero-diffusivity equation. The zero-diffusivity equation---in contrast to the positive diffusivity equation---is not straightforwardly well-posed, and in the case of correlated-in-time advection-diffusion equation~\eqref{eq:advection-diffusion-free} with a $C^\alpha_x$ advecting flow, it is generically ill-posed, admitting many solutions even under various selection principles~\cite{colombo_anomalous_2023}. However, for the white-in-time equation~\eqref{eq:kraichnan-free}, the zero-diffusivity equation is well-posed~\cite{jan_integration_2002,jan_flows_2004,drivas_anomalous_2025}, so one could seek to make statements directly about this zero-diffusivity equation. We choose not to pursue this path for a few reasons. Dealing with the positive diffusivity case is technically more straightforward as the well-posedness theory is standard and well-behaved in both the advection-diffusion case and the white-in-time case. One can hope to transform uniform-in-diffusivity results to the zero-diffusivity equation under some weak convergence results. Finally, the physics literature often phrases its investigation in terms of uniform in (small enough) diffusivity estimates, so this setting fits well within the literature.

\section{Overview of the argument}
\label{s:overview}

We now discuss the argument, which breaks into two independent main pieces: 
\begin{itemize}
    \item The proof of the general result of anomalous dissipation and regularization for white-in-time velocity fields given by Theorem~\ref{thm:kraichnan}: discussed in Section~\ref{s:anomalous overview} and given in Section~\ref{s:anomalous}, relying on a lattice inequality proved in Section~\ref{s:lattice}.
    \item The proof of the Obukhov--Corrsin bounds of Theorem~\ref{thm:correlated lower bound} and Theorem~\ref{thm:white lower bound} under the assumption of anomalous dissipation and regularization: discussed in Section~\ref{s:obukhov corrsin overview} and given in Section~\ref{s:lower bounds}.
\end{itemize}  

In addition to these two main steps, we need to conclude Theorem~\ref{thm:anomalous dissipation and regularization special} and Theorem~\ref{thm:OC special} as special cases of the Theorem~\ref{thm:kraichnan} and Theorem~\ref{thm:white lower bound} respectively, as well as provide the short arguments for Corollary~\ref{cor:infinite smoothing in the z direction}, Proposition~\ref{prop:invariant measures}, and Proposition~\ref{prop:OC upper bounds}. These straightforward arguments are provided at the end of this section and won't be further discussed.

\subsection{Overview of the proof of anomalous regularization and anomalous dissipation in white-in-time models}
\label{s:anomalous overview}

The proof of Theorem~\ref{thm:kraichnan} is strongly inspired by the argument of~\cite{luo_elementary_2024} as well as~\cite{rowan_anomalous_2024}. Let us fix $\alpha \in (0,1)$ in this discussion and consider the velocity field $du_t$ with Fourier coefficients defined by Definition~\ref{defn:Kraichnan standard}. We then let $\phi^\kappa_t$ solve~\eqref{eq:kraichnan-free}.

In~\cite{rowan_anomalous_2024}, we consider
\[
    g^\kappa_t(x) := \int \E \phi^\kappa_t(y) \phi^\kappa_t(y+x)\,dy,\]
    which satisfies the equation
\begin{equation}
\label{eq:g equation overview}
    \dot g^\kappa_t= 2 \kappa \Delta g + \nabla \cdot a \nabla g^\kappa_t
\end{equation}
for a matrix $a$ satisfying
\[v \cdot a(x) v \geq |x|^{2\alpha} |v|^2,\] 
hence~\eqref{eq:g equation overview} becomes a degenerate parabolic equation as $\kappa \to 0.$
We note we also consider $g^\kappa$ defined this way in Section~\ref{s:oc bounds white}. A direct computation then verifies that
\[\frac{d}{dt} \frac{1}{2} \|g^\kappa\|_{L^2}^2 \leq - \||x|^{\alpha} \nabla g^\kappa\|_{L^2}^2 \leq -C^{-1} \|g^\kappa\|_{H^{1-\alpha}}^2,\]
where for the second inequality we use weighted Sobolev inequality~\cite{caffarelliFirstOrderInterpolation1984}. Integrating in time, we thus get that
\begin{equation}
\label{eq:g inequality}
\int_0^1 \|g^\kappa_t\|_{H^{1-\alpha}}^2\,dt \leq \|g^\kappa_0\|_{L^2}^2.
\end{equation}
This looks a lot like our desired anomalous regularization estimate~\eqref{eq:anomalous regularization special}, except that the bound is on $g^\kappa_t$ instead of $\phi^\kappa_t$. This estimate is much weaker, as $g^\kappa$ is already integrated in $x$ compared to $\phi^\kappa$; in fact, $L^\infty$-type estimates on $g^\kappa$ are equivalent to $L^2$-type estimates on $\phi^\kappa$. Thus we want the inequality~\eqref{eq:g inequality} in $L^\infty_x$ instead of $L^2_x$. However, it is not particularly clear how to get an $L^\infty$-type smoothing estimate for the degenerate parabolic equation~\eqref{eq:g equation overview}, and the general theory of degenerate parabolic equations is not very well developed (see~\cite[Remark 1.4]{drivas_anomalous_2025} for a discussion of the relevant literature on degenerate PDE regularity theory).

Following the lead of~\cite{luo_elementary_2024}, we can look at the problem in Fourier space, which allows one to compute---using the Fourier transform of~\eqref{eq:g equation overview}---that for any $p \geq 1$
\begin{equation}
\label{eq:ell p inequality overview}
\frac{d}{dt} \sum_k a_k^p \leq -2\pi^2p \sum_{k,j \in \Z^d \backslash\{0\}} w_j^2  |\Pi_{j^\perp} k|^2(a_{k+j}^{p-1} - a_k^{p-1})(a_{k+j}- a_k),
\end{equation}
where $a_k(t) := \E |\hat \phi^\kappa_t(k)|^2.$ Taking $p \downarrow 1$ in Fourier space is analogous to taking $p \uparrow \infty$ in real space, thus if we can get a good ``weighted Poincar\'e-type'' inequality for all $p>1$ (in place of a real-space weighted Sobolev inequality), we can hope to get an (almost) $L^\infty(\T^d)$ version of~\eqref{eq:g inequality}. We note that we cannot work directly with $p=1$ as in that case the term on the right hand side of~\eqref{eq:ell p inequality overview} vanishes, which is the Fourier space version of the fact that the velocity field $du_t$ doesn't appear directly in the time derivative of $\|\phi_t\|_{L^2(\T^d)}^2$. This formulation of the problem has the advantage of utilizing additional structure of the problem, since the appearance of positive weight $w_j^2$ on the right hand side is related to the fact that covariance functions are positive definite, hence have positive Fourier transforms. In the real space formulation of the problem~\eqref{eq:g equation overview}, we have that $a(x)$ has a negative Fourier transform, but it is not clear how to use this information.

The requisite Poincar\'e inequality is provided by the following proposition, which is a direct corollary of Lemma~\ref{lem:ell p type inequality with explicit constants} proved in Section~\ref{s:lattice}.

\begin{proposition}
\label{prop:ell p inequality clean}
        Let $\alpha \in (0,1)$, $(w_k)_{k \in \Z^d \backslash\{0\}}$ satisfying Assumption~\ref{asmp:wk} for $\alpha$, and $(a_k)_{k \in \Z^d \backslash \{0\}}$ such that $a_k \geq 0 $ and 
        \begin{equation*}
        \sum_{k \in \Z^d \backslash\{0\}} |k|^{2(1-\alpha)} a_k^p <\infty.
    \end{equation*}
    Then for all $R \geq r_0$, there exists $C(\delta, R, \Psi)>0$ so that for $S$ defined by~\eqref{eq:S-def}, for all $p \in (1,2],$ we have the bound
    \[\sum_{k \in \Z^d \backslash \{0\}} S(R |k|)^2 a_k^p \leq \frac{C}{p-1} \sum_{k,j \in \Z^d \backslash\{0\}} w_j^2  |\Pi_{j^\perp} k|^2(a_{k+j}^{p-1} - a_k^{p-1})(a_{k+j}- a_k).\]
\end{proposition}

The proof of this inequality is further discussed in Section~\ref{s:lattice}. Combining Proposition~\ref{prop:ell p inequality clean} with~\eqref{eq:ell p inequality overview} and integrating in time, we get that for $p>1$
\[\int_0^1 \sum_{k \in \Z^d \backslash \{0\}} S(R |k|)^2 a_k^p \,dt \leq \frac{C}{p-1} \|a_k(0)\|_{\ell^1}^p = \frac{C}{p-1} \|\phi^\kappa_0\|_{L^2}^{2p}.\]
If we had the above inequality with $p =1$ (and a finite constant on the right hand side), by recalling the definition of $a_k(t)$, we could directly conclude the smoothing estimate of Theorem~\ref{thm:kraichnan} (or an even stronger version, without the logarithmic corrections). However, it turns out that by a careful application of H\"older's inequality with appropriately chosen exponents on doubly exponential annuli, the above inequalities for $p>1$ actually imply the desired smoothing estimate for $p=1$, after introducing additional logarithmic weights.

We have thus covered the smoothing estimate aspect of Theorem~\ref{thm:kraichnan}, but we also need the anomalous dissipation aspect. That is now straightforward, since using that $S(R|k|) \geq S(R) \geq 1$, we can combine Proposition~\ref{prop:ell p inequality clean} and~\eqref{eq:ell p inequality overview} to give that for all $p >1$
\[\frac{d}{dt} \sum_k a_k^p \leq - C^{-1} (p-1)  \sum_k a_k^p.\]
Applying then Gr\"onwall's inequality, we have exponential decay of all $\ell^p$ norms of $a_k$ for $p>1$. Again we want the decay for $p=1$, but this then follows using the regularization estimate, H\"older's inequality, and the lower bound on $S$ given by~\eqref{eq:S not too small} in order to bound the $\ell^1$ norm by the $\ell^p$ norm for some $p>1$, thus concluding the proof of Theorem~\ref{thm:kraichnan}.

\subsection{Overview of the proof of the Obukhov--Corrsin lower bounds}
\label{s:obukhov corrsin overview}

Let us only discuss here the proof of Theorem~\ref{thm:correlated lower bound}. The proof of Theorem~\ref{thm:white lower bound} follows similarly, working the equation for $g^\kappa$~\eqref{eq:g equation overview} for the commutator bounds instead of directly with~\eqref{eq:advection-diffusion-free}. 

The proof of Theorem~\ref{thm:correlated lower bound} is inspired by the proof of the Batchelor spectrum bounds~\cite{bedrossian_batchelor_2021}, the refinement of this argument in~\cite{cooperman_fourier_2025}, and the commutator argument of~\cite{constantin_onsagers_1994} (appearing in the advection-diffusion setting first in~\cite{drivas_anomalous_2022}). The idea is to perform the commutator bound of~\cite{constantin_onsagers_1994} on $\phi^\kappa_t$ solving~\eqref{eq:advection-diffusion-free}, which yields for all $r \in (0,1)$,
\[ {-}\frac{d}{dt} \|\eta_r* \phi^\kappa_t\|_{L^2}^2 \leq C \big(\kappa r^{-2}  + \|v_t\|_{C^\alpha_x} r^{-(1-\alpha)} \big)  \frac{1}{|B_r|} \int_{B_r} \big\|\phi_t^\kappa(x-y) - \phi_t^\kappa(x)\big\|_{L^2_x}^2\,dy.
\]
where $\eta_r$ is a standard family of mollifiers. Using that $\phi^\kappa_t \to 0$ as $t \to \infty$, we can integrate the left hand side over $[0,\infty)$ to get an $O(1)$ constant. Restricting the range of $r$ so that $\kappa r^{-2} \leq r^{-(1-\alpha)}$ and bounding the right hand side in Fourier space, we then get
\[ Cr^{1-\alpha}  \leq  \sum_{k \in \Z^d \backslash\{0\}} \min(|k|^2 r^2, 1)  \E_{\mu^\kappa} | \hat \theta^\kappa(k)|^2,\]
where we Proposition~\ref{prop:invariant measures} to relate time integrals of $\phi^\kappa_t$ to expectations over the invariant measure $\mu^\kappa$ of~\eqref{eq:advection-diffusion-forced}. We then note that this would give us exactly the desired lower bound on the energy spectrum if we could show that the dominant contribution to the sum on the right hand side was given by a geometrically annulus centered at $r^{-1}, \{ C^{-1} r^{-1} \leq |k| \leq C r^{-1}\}.$

This same problem, of having to localize the contribution of the sum in Fourier space, appeared in~\cite{cooperman_fourier_2025}. The idea there is to utilize any available the upper bounds on $\E_{\mu^\kappa} |\hat \theta^\kappa(k)|^2$---in that case provided by exponential mixing, here provided by the regularity of the invariant measure given by Proposition~\ref{prop:OC upper bounds}---to localize the sum. To that end, we split the sum into sums over $|k| <a, a \leq |k| \leq b$, and $|k| >b$ for $1 \leq a < r^{-1} < b <\infty$. We then use upper bounds provided by Proposition~\ref{prop:OC upper bounds} to control the contributions of the $|k| <a$ and $|k| >b$ parts, then choosing $a$ small enough and $b$ large enough to reabsorb these errors on the left hand side. 

The result is a lower bound on the mass on the annulus $\{a \leq |k| \leq b\}$. Ideally, we want $a = C^{-1} r^{-1}, b = C r^{-1}$. However, this would only be attainable if we had regularity estimates on the invariant measure all the way up to the endpoint of $H^{\frac{1-\alpha}{2}}_x$, without any logarithmic corrections. Due to the logarithmic corrections even in the best case, $a$ is logarithmically in $r^{-1}$ smaller than $C^{-1} r^{-1}$ and $b$ logarithmically bigger than $C r^{-1}.$ Thus we only get the bounds on (at best) slightly super-geometrically large annuli in Fourier space. In general, if we have even worse smoothing estimates, algebraically below $H^{\frac{1-\alpha}{2}}_x$---that is only $H^{\beta}_x$ for some $\beta < \frac{1-\alpha}{2}$---we have even larger annuli and our estimates are substantially worse.

\subsection{Proofs of Propositions~\ref{prop:invariant measures} and~\ref{prop:OC upper bounds}}

\begin{proof}[Proof of Proposition~\ref{prop:invariant measures}]
    We consider only the case of~\eqref{eq:advection-diffusion-forced}; the case of~\eqref{eq:kraichnan-forced} follows even more straightforwardly as in that case we don't need to track $(v_{s+t})_{s \in \R}$. We fix $\kappa>0$ throughout this argument. Let $\theta^\kappa_t$ solve~\eqref{eq:advection-diffusion-forced} with arbitrary initial data $\theta_0 \in L^2(\T^d)$ with $\int \theta_0(x)\,dx =0$. For some trajectory $(v_r)_{r \in \R} \in C^0(\R \times \T^d)$ and $-\infty < s \leq t <\infty$, let $\sol_{s,t}^{v}$ be the solution operator to the unforced advection-diffusion equation~\eqref{eq:advection-diffusion-free} from time $s$ to time $t$. Then by Duhamel's principle, we have that for any $t>0$
    \[\theta^\kappa_t = \sol_{0,t}^v \theta_0 + \int_0^t \sol^v_{s,t} F\, dW_s=  \sol_{-t,0}^{(v_{s+t})_{s \in \R}} \theta_0 + \int_0^t \sol^{(v_{s+t})_{s \in \R}}_{-(t-s),0} F\, dW_s.\]
    We then note that by the standard energy estimate for an advection-diffusion equation and the Poincar\'e inequality, $\| \sol_{-t,0}^{(v_{s+t})_{s \in \R}} \theta_0\| \leq e^{-C^{-1} \kappa t} \|\theta_0\|_{L^2}.$ We also have by the time-stationarity in law of $v$,
    \[\int_0^t \sol^{(v_{s+t})_{s \in \R}}_{-(t-s),0} F\, dW_s \stackrel{d}{=} \int_0^t \sol^{v}_{-s,0} F\, dW_s,\]
    where $\stackrel{d}{=}$ means equality in distribution (or law). Thus sending $t \to \infty$, we have that, in law,
    \[\lim_{t \to \infty} \theta^\kappa_t \stackrel{d}{=} \int_0^\infty \sol^{v}_{-s,0} F\, dW_s,\]
    where the integral on the right hand side makes sense again using the basic decay estimate of the advection-diffusion equation equation to give that surely, $\|\sol^{v}_{-s,0} F\|_{L^2} \leq e^{- C^{-1} \kappa t} \|F\|_{L^2}.$ 

    This above argument essentially immediately implies that the random variable
    \[\big(v, \int_0^\infty \sol^v_{-s,0} F\,dW_s\big)\]
    has a law $\mu^\kappa$ that is invariant for the Markov process $((v_{s+t})_{s \in \R}, \theta^\kappa_t)$ for $\theta^\kappa_t$ solving~\eqref{eq:advection-diffusion-forced}, and further that it is the unique invariant measure for the process also satisfying $\mu^\kappa(dv, L^2(\T^d)) = \nu(dv)$ (since the dependency on the data $\theta_0$ vanishes as $t \to \infty$). An application of It\^o's isometry immediately implies~\eqref{eq:equilibrium equals free decay}, allowing us to conclude.
\end{proof}

\begin{proof}[Proof of Proposition~\ref{prop:OC upper bounds}]
    By Proposition~\ref{prop:invariant measures}, we have that
    \begin{equation}
    \label{eq:invariant display upperbound}
    \E_{\mu^\kappa} \|\theta^\kappa\|^2_{\sigma(H)} = \int_0^\infty \E \|\phi^\kappa_t\|_{\sigma(H)}^2\,dt.
    \end{equation}
    By the definition of anomalous dissipation and anomalous regularization, we have for any $n \in \N$, 
    \[\int_n^{n+1} \|\phi^\kappa_t\|_{\sigma(H)}^2\,dt \leq C\|\phi^\kappa_n\|_{L^2}^2 \leq C e^{-C^{-1} n} \|F\|_{L^2(\T^d)}^2.\]
    Summing this inequality over $n \in \N$ and using~\eqref{eq:invariant display upperbound}, we conclude.
\end{proof}

\subsection{Proofs of Theorem~\ref{thm:anomalous dissipation and regularization special}, Theorem~\ref{thm:OC special}, and Corollary~\ref{cor:infinite smoothing in the z direction}}
\label{ss:theorems to theorems}

\begin{proof}[Proof of Theorem~\ref{thm:anomalous dissipation and regularization special}]
    We want to apply Theorem~\ref{thm:kraichnan}. As such, we need to compute $S(r)$ associated to the $(w_k)_{k \in \Z^d \backslash\{0\}}$ as defined by Definition~\ref{defn:Kraichnan standard} and Definition~\ref{defn:Kraichnan shear} and verify Assumption~\ref{asmp:wk}. We note first that~\eqref{eq:C alpha condition} is satisfied by construction. 

    First for the case of Definition~\ref{defn:Kraichnan standard}, we compute that for $r \geq r_0(\alpha)$
    \begin{align*} S(r) &= Z^{-1}\sum_{1 \leq |j| \leq r} |j|^{1+\alpha} \frac{|j|^{-d - 2\alpha}}{(\log |j| +1)^2} 
    \\&\approx \int_{r_0/2}^r \frac{s^{-\alpha}}{(\log s)^2}\,ds
    \\&\approx \int_{r_0/2}^r (1- \alpha) \frac{s^{1-\alpha}}{(\log s)^2} -2 \frac{s^{1-\alpha}}{(\log s)^3}\,ds
    \\&=  \frac{r^{1-\alpha}}{(\log r)^2} -  \frac{(r_0/2)^{1-\alpha}}{(\log(r_0/2))^2}
    \\&\approx   \frac{r^{1-\alpha}}{(\log r)^2},
    \end{align*}
    where by $A \approx B$ we mean that there exists $C(d,\alpha)>0$ such that  $C^{-1} A \leq B \leq CA$. Similarly, for the case of Definition~\ref{defn:Kraichnan shear}, we compute for $r \geq r_0(\alpha)$:
    \[
        S(r) = 4 Z^{-1} \sum_{j=1}^r j^{1+\alpha} \frac{j^{-1 - 2\alpha}}{(\log j +1)^2} \approx \int_{r_0/2}^r \frac{s^{-\alpha}}{(\log s)^2}\,ds\approx   \frac{r^{1-\alpha}}{(\log r)^2}.\]
    Condition~\eqref{eq:S not too small} and~\eqref{eq:geometric fluctuation bound} of Assumption~\ref{asmp:wk} follow then directly for both cases of Definition~\ref{defn:Kraichnan standard} and Definition~\ref{defn:Kraichnan shear}. The final condition~\eqref{eq:nondegenerate} also follows from the straightforward computation that for each case there exists $\delta>0$ such that for any $v \in \R^d$ and any $r \in (0,\infty)$,
    \[\sum_{|j| =r} w_j^2 |\Pi_{j^\perp} v| \geq \delta \sum_{|j| =r} w_j^2.\]
    Thus we can apply Theorem~\ref{thm:kraichnan} which directly gives the result, after using our bounds on $S(r)$. 
\end{proof}

Theorem~\ref{thm:OC special} is a direct consequence of Theorem~\ref{thm:anomalous dissipation and regularization special} and Theorem~\ref{thm:white lower bound} with $\beta = 1- \alpha$ and $m=2$.

\begin{proof}[Proof of Corollary~\ref{cor:infinite smoothing in the z direction}]
    Fix $3 \leq j \leq d$ and let $|\partial_j| := (-\partial_j \partial_j)^{1/2}$, where throughout we are using the functional calculus of self-adjoint operators on $L^2(\T^d)$. Then, since $du_t$ is invariant under $x_j$ translations, we have that for any $s>0$, $|\partial_j|^s \phi^\kappa_t$ solves~\eqref{eq:kraichnan-free}. Thus by~\eqref{eq:anomalous dissipation special} there exists $C>0$ such that for any $t \geq 0$
    \[\E \||\partial_j|^s\phi^\kappa_t\|_{L^2}^2 \leq C\||\partial_j|^s \phi^\kappa_0\|_{L^2}^2 \]
    and by~\eqref{eq:anomalous regularization special}, 
    \[\E \int_0^1 \||\partial_j|^{s + \frac{1-\alpha}{2}} |\hat \phi^\kappa_t(k)|^2\,dt \leq C\E \int_0^1 \frac{|k|^{2(1-\alpha)}}{(\log |k| + 1)^4} \big| \widehat{|\partial_j|^s \phi^\kappa_t}(k)\big|^2\,dt \leq C \||\partial_j|^s \phi^\kappa_0\|_{L^2}^2.\]
    Combining the above two displays, we get that for all $t>0$
    \[\E \||\partial_j|^{s+ \frac{1-\alpha}{2}}\phi^\kappa_t\|_{L^2}^2 \leq C t^{-1}\||\partial_j|^s \phi^\kappa_0\|_{L^2}^2.\]
    Iterating this inequality, we conclude the result.
\end{proof}

\section{Anomalous regularization and dissipation for white-in-time models}
\label{s:anomalous}

We fix $\alpha \in (0,1)$ and Fourier coefficients $w_k$ for $du_t$ as given in~\eqref{eq:dut def} such that $w_k$ satisfy Assumption~\ref{asmp:wk} for $\alpha$.

The following computations for Proposition~\ref{prop:ak equation} and Proposition~\ref{prop:ell p energy inequality} essentially appear in~\cite[Sections 2-3]{luo_elementary_2024}; we repeat them here for the reader's convenience.

\begin{proposition}
\label{prop:ak equation}
    Let $\kappa >0, F \in L^2(\T^d)$ with $\int F(x)\,dx =0$, and $\phi^\kappa_t$ solve~\eqref{eq:kraichnan-free}. The for $(a_k)_{k \in \Z^d \backslash \{0\}}$ with $a_k \geq 0$ defined by 
    \[a_k(t) := \E |\hat \phi^\kappa_t(k)|^2,\]
    the $a_k$ solve the equation
    \begin{equation}
    \label{eq:ak equation}
    \dot a_k = - 8 \pi^2 \kappa |k|^2 a_k -4\pi^2 \sum_j w_j^2 |\Pi_{j^\perp} k|^2 \big(a_k -a_{k-j}\big).
    \end{equation}
\end{proposition}

\begin{proof}
    Taking the Fourier transform of~\eqref{eq:kraichnan-free}, we have that
    \[\dot{\hat \phi^\kappa_t}(k) = - 4\pi^2 \kappa |k|^2 \hat \phi^\kappa_t(k)  - 2\pi i k \cdot \sum_j  w_j  \hat \phi^\kappa_t(k-j)  \sum_{n=1}^{d-1}v_{j,n} \circ dW_t^{j,n}.\]
    Thus 
    \begin{align}
    \frac{d}{dt} \E | \hat \phi^\kappa_t(k)|^2 &= - 8 \pi^2 \kappa |k|^2 \E | \hat \phi^\kappa_t(k)|^2
 \notag \\&\qquad- 2\pi i k \cdot \sum_j  \sum_{n=1}^{d-1}v_{j,n}  w_j \E \hat \phi^\kappa_t(k-j) \overline{\hat \phi^\kappa_t(k)}  \circ dW_t^{j,n} 
 \notag \\&\qquad+ 2\pi i k \cdot \sum_j  \sum_{n=1}^{d-1}v_{j,n}  w_j \E \overline{\hat \phi^\kappa_t(k-j)} \hat \phi^\kappa_t(k)  \circ \overline{dW_t^{j,n}}.
        \label{eq:fourier time derivative}
    \end{align}
    We then note that 
    \begin{align*}
    &\E \hat \phi^\kappa_t(k-j) \overline{\hat \phi^\kappa_t(k)}  \circ dW_t^{j,n} = \frac{1}{2} \big[  d(\hat \phi^\kappa_t(k-j) \overline{\hat \phi^\kappa_t(k)}), dW^{j,n}_t\big]
    \\&\qquad=\E \pi i \hat \phi^\kappa_t(k-j) k \cdot \sum_\ell  w_\ell \overline{\hat \phi^\kappa_t}(k-\ell)  \sum_{m=1}^{d-1}v_{\ell,m} \big[ \overline{dW_t^{\ell,m}}, dW^{j,n}_t\big]
       \\&\qquad \qquad  -\E \pi i\overline{\hat \phi^\kappa_t(k)}  (k-j) \cdot \sum_\ell  w_\ell  \hat \phi^\kappa_t(k-j-\ell)  \sum_{m=1}^{d-1}v_{\ell,m}  \big[  dW_t^{\ell,m} , dW^{j,n}_t\big]
        \\&\qquad=\pi i\E  |\hat \phi^\kappa_t|^2(k-j)  w_j k \cdot  v_{j,n} - \pi i\E|\hat \phi^\kappa_t|^2(k)   w_{j}k \cdot  v_{j,n},
    \end{align*}
    where we use $w_j = w_{-j}, v_{j,n} = v_{-j,n}, [dW^{\ell,m}_t, dW^{j,n}_t] = \delta_{\ell,-j} \delta_{m,n}$, $[\overline{dW^{\ell,m}_t}, dW^{j,n}_t] = \delta_{\ell,j} \delta_{m,n},$ and $j \cdot v_{j,m}= 0.$
    Thus
    \begin{align*}
        &- 2\pi i k \cdot \sum_j  \sum_{n=1}^{d-1}v_{j,n}  w_j \E \Big(\hat \phi^\kappa_t(k-j) \overline{\hat \phi^\kappa_t(k)}  \circ dW_t^{j,n} - \overline{\hat \phi^\kappa_t(k-j)} \hat \phi^\kappa_t(k)  \circ \overline{dW_t^{j,n}}\Big)
        \\&\qquad= - 2\pi i k \cdot  \sum_j  \sum_{n=1}^{d-1}v_{j,n}  w_j \Big(\pi i\E  |\hat \phi^\kappa_t|^2(k-j)  w_j k \cdot  v_{j,n} - \pi i\E|\hat \phi^\kappa_t|^2(k)   w_{j}k \cdot  v_{j,n}
        \\&\qquad\qquad\qquad\qquad\qquad\qquad\quad- \overline{\pi i\E  |\hat \phi^\kappa_t|^2(k-j)  w_j k \cdot  v_{j,n} - \pi i\E|\hat \phi^\kappa_t|^2(k)   w_{j}k \cdot  v_{j,n}}\Big) 
        \\&\qquad= -4\pi^2 \sum_j w_j^2 \Big(  \E|\hat \phi^\kappa_t|^2(k)-\E  |\hat \phi^\kappa_t|^2(k-j)   \Big) \sum_{n=1}^{d-1} (k\cdot v_{j,n})^2   \\&\qquad=  -4\pi^2 \sum_j w_j^2 |\Pi_{j^\perp} k|^2 \Big(\E|\hat \phi^\kappa_t|^2(k) -\E  |\hat \phi^\kappa_t|^2(k-j)\Big).
    \end{align*}
    Plugging this into~\eqref{eq:fourier time derivative}---and recalling the definition of $a_k,f_k$---we conclude.
\end{proof}

The following computation is a bit formal but can easily be verified to hold due to the presence of the diffusion $\kappa>0$, ensuring that for any $t \geq 0$
\[\int_0^t \sum_k |k|^2|a_k(s)|\,ds = \E \int_0^t \|\phi^\kappa_s\|_{H^1(\T^d)}^2\,ds \leq \kappa^{-1} \|\phi^\kappa_0\|_{L^2}^2 < \infty.\]
This (together with H\"older's inequality) ensures all of our sums are absolutely convergent, hence the below manipulations are valid.

\begin{proposition}
\label{prop:ell p energy inequality}
    For all $p \geq 1$, letting $(a_k)_{k \in \Z^d \backslash 0}$ with $a_k \geq 0$ and $\sum_{k \in \Z^d \backslash\{0\}} a_k(0) <\infty$ such that $a_k$ solves~\eqref{eq:ak equation}. Then 
    \begin{equation}
    \label{eq:ell p energy equality}
    \frac{d}{dt} \sum_k a_k^p = - 8 \pi^2 \kappa p\sum_k |k|^2 a_k^p -2\pi^2p \sum_{k,j} w_j^2 |\Pi_{j^\perp} k|^2 \big(a_k -a_{k-j}\big) \big(a_k^{p-1} - a_{k-j}^{p-1}\big).
    \end{equation}
\end{proposition}

\begin{proof}
    Computing directly with~\eqref{eq:ak equation}, we have that
    \begin{equation}
\label{eq:dt of ell p}
        \frac{d}{dt} \sum_k a_k^p = p \sum_k a_k^{p-1} \dot a_k
        =- 8 \pi^2 \kappa p\sum_k |k|^2 a_k^p -4\pi^2p \sum_{k,j} w_j^2 |\Pi_{j^\perp} k|^2 \big(a_k -a_{k-j}\big) a_k^{p-1}.
    \end{equation}
    Then we note that, using that $w_j = w_{-j}$,
    \begin{align*}
        \sum_{k,j} w_j^2 |\Pi_{j^\perp} k|^2 \big(a_k -a_{k-j}\big) a_k^{p-1} &= \sum_{k,j} w_j^2 |\Pi_{j^\perp} (k+j)|^2 \big(a_{k+j} -a_{k}\big) a_{k+j}^{p-1}
        \\&= \sum_{k,j} w_{-j}^2 |\Pi_{j^\perp} k|^2 \big(a_{k-j} -a_{k}\big) a_{k-j}^{p-1}
        \\&= -\sum_{k,j} w_{j}^2 |\Pi_{j^\perp} k|^2 \big(a_{k}-a_{k-j} \big) a_{k-j}^{p-1}.
    \end{align*}
    Together with~\eqref{eq:dt of ell p}, this gives the result.
\end{proof}

We now want to apply Proposition~\ref{prop:ell p inequality clean} together with Proposition~\ref{prop:ell p energy inequality} to deduce the anomalous dissipation and anomalous regularization estimates of Theorem~\ref{thm:kraichnan} (in the language of the $a_k$). We recall that we are assuming the coefficients $w_k$ of $du_t$ satisfy Assumption~\ref{asmp:wk} for $\alpha \in (0,1).$ We emphasize the constants in the result and proof below are uniform in $\kappa$. Theorem~\ref{thm:kraichnan} is directly implied by the result below, after recalling the definition of $a_k.$

\begin{proposition}
    Let $(a_k)_{k \in \Z^d \backslash 0}$ with $a_k \geq 0$ and $\sum_{k \in \Z^d \backslash \{0\}} a_k(0) <\infty$ such that $a_k$ solves~\eqref{eq:ak equation}. Then for all $R \geq r_0,$ there exists $C(\beta, \delta,R,\Psi)>0$ such that for all $\kappa \in (0,1],$ we have the bounds
    \begin{align}
        \label{eq:exponential decay}
        \|a_k(t)\|_{\ell^1(\Z^d)} &\leq e^{{-} C^{-1}t}\|a_k(0)\|_{\ell^1(\Z^d)},\\
        \label{eq:smoothing}
        \int_0^1 \sum_{k \in \Z^d \backslash \{0\}} \Big(\frac{S(R|k|)}{\log |k| + 1}\Big)^2 a_k(s)\,ds  &\leq C  \|a_k(0)\|_{\ell^1(\Z^d)}.
    \end{align}
\end{proposition}

\begin{proof}
    As noted above, since $\kappa>0$, we have that for all $t \geq 0$
    \[\int_0^t \sum_k |k|^2 |a_k(s)|\,ds < \infty.\]
    Thus we can apply~\eqref{eq:ell p energy equality} and Proposition~\ref{prop:ell p inequality clean} to give that for $p \in (1,2],$
    \begin{align}
    \frac{d}{dt} \sum_k a_k^p &\leq -2\pi^2p \sum_{k,j} w_j^2 |\Pi_{j^\perp} k|^2 \big(a_k -a_{k-j}\big) \big(a_k^{p-1} - a_{k-j}^{p-1}\big)
    \notag \\& \leq -C^{-1} (p-1)\sum_{k} S(R |k|)^2 a_k^p.
    \label{eq:inequalities combined}
    \end{align}

    For~\eqref{eq:smoothing}, we integrate~\eqref{eq:inequalities combined} over $t \in [0,1],$
    \begin{equation}
   \label{eq:intergrated bound} 
    \int_0^1\sum_{k \in \Z^d \backslash \{0\}} S(R |k|)^2 a_k^p(s)\,ds \leq \frac{C}{p-1} \sum_k a_k^p(0) \leq \frac{C}{p-1} \|a_k(0)\|_{\ell^1(\Z^d)}^p.
    \end{equation}
    We however want the bound in $\ell^1$ on the left hand side, which we can get from~\eqref{eq:intergrated bound} by the following argument. By H\"older's inequality, for any sequence $p_n \in (1,2)$,
    \begin{align*}&\int_0^1 \sum_{k \in \Z^d \backslash \{0\}} (\log |k| + 1)^{-2} S(R|k|)^2 a_k(s)\,ds 
    \\&\qquad= \int_0^1\sum_{n=0}^\infty \sum_{ 2^{2^n-1} \leq |k| < 2^{2^{n+1}-1}} (\log |k| + 1)^{-2} S(R|k|)^2 a_k(s)\,ds
    \\&\qquad\leq \sum_{n=0}^\infty  2^{-2 n} 2^{p_n^{-1}(p_n-1)2^{n+1}} \Big(\int_0^1\sum_{ 2^{2^n-1} \leq |k| < 2^{2^{n+1}-1}}  S(R|k|)^{2p_n} a_k^{p_n}(s)\,ds\Big)^{1/p_n}.
    \end{align*}
    We then use that $S$ is increasing to bound the final term in the above display by 
    \begin{align*}
     \\&\sum_{n=0}^\infty  2^{(p_n-1)2^{n+1}-2 n} S(R 2^{2^{n+1}})^{2 (p_n-1)} \Big(\int_0^1\sum_{ 2^{2^n-1} \leq |k| < 2^{2^{n+1}}}  S(R|k|)^{2} a_k^{p_n}(s)\,ds\Big)^{1/p_n}
    \\&\qquad \leq \|a_k(0)\|_{\ell^1(\Z^d)} \sum_{n=0}^\infty  2^{(p_n-1)2^{n+1}-2 n}  (R2^{2^{n+1}})^{2(1-\alpha)(p_n-1)}\Big( \frac{C}{p_n-1}\Big)^{1/p_n},
    \end{align*}
    where for the final ienquality, we use~\eqref{eq:intergrated bound} to bound the $ds$ integral and~\eqref{eq:S bar bound} to bound $S(R2^{2^n})$.
    Since $p_n \in (1,2)$, we have that $ \Big( \frac{C}{p_n-1}\Big)^{1/p_n} \leq   \frac{C}{p_n-1}.$ Using also that $\alpha \geq 0, p_n \leq 2$, putting everything together, we see 
    \[\int_0^1 \sum_{k \in \Z^d \backslash \{0\}}  \frac{S(R|k|)^2}{(\log |k| + 1)^{2}} a_k(s)\,ds  \leq C \|a_k(0)\|_{\ell^1(\Z^d)} \sum_{n=0}^\infty  2^{3(p_n-1)2^{n+1}-2 n} (p_n-1)^{-1}.\]
    We choose then $p_n = 2^{-n} + 1$, giving
    \[\int_0^1 \sum_{k \in \Z^d \backslash \{0\}}  \frac{S(R|k|)^2}{(\log |k| + 1)^{2}} a_k(s)\,ds  \leq  C   \|a_k(0)\|_{\ell^1(\Z^d)} \sum_{n=0}^\infty  2^{-n}  \leq C \|a_k(0)\|_{\ell^1(\Z^d)},\]
    which is~\eqref{eq:smoothing}.

    We now consider~\eqref{eq:exponential decay}. By~\eqref{eq:inequalities combined} and~\eqref{eq:S not too small}, we have that
    \[\frac{d}{dt} \sum_k a_k^p \leq  -C^{-1} (p-1) \sum_k a_k^p,\]
    so we can apply Gr\"onwall's inequality to give for any $t \geq 0$
    \begin{equation}
    \label{eq:ell p exponential decay}
    \|a_k(t)\|_{\ell^p} \leq e^{- C^{-1} (p-1) t} \|a_k(0)\|_{\ell^p} \leq  e^{- C^{-1} (p-1) t}\|a_k(0)\|_{\ell^1}. 
    \end{equation}
    Then letting $\beta$ as in Assumption~\ref{asmp:wk}, by H\"older's inequality we have,
    \[\|a_k\|_{\ell^1} \leq \||k|^\beta a_k\|_{\ell^1}^{1/2}\||k|^{-\beta} a_k\|_{\ell^1}^{1/2} \leq  \||k|^\beta a_k\|_{\ell^1}^{1/2}  \|a_k\|_{\ell^p}^{1/2}\||k|^{-\beta}\|_{\ell^q}^{1/2}  ,\]
    for $p,q$ H\"older conjugates. Then we can choose $p(\beta) \in (1,2]$ so that $\||k|^{-\beta}\|_{\ell^q}\leq C(\beta)<\infty$, thus
    \[\|a_k\|_{\ell^1} \leq C \||k|^\beta a_k\|_{\ell^1}^{1/2}  \|a_k\|_{\ell^p}^{1/2}.\]
    Then by~\eqref{eq:ell p exponential decay}, there exists $C(\beta,\delta,R,\Psi)>0$ so that for any $t \geq 0$
    \[\|a_k(t)\|_{\ell^1} \leq C e^{- C^{-1} t} \||k|^\beta a_k(t)\|_{\ell^1}^{1/2} \|a_k(0)\|_{\ell^1}^{1/2}.\]
    Integrating this bound over $s \in [t,t+1],$ we have that
    \begin{align}
    \int_t^{t+1} \|a_k(s)\|_{\ell^1}\,ds & \leq  C
    e^{- C^{-1} t}  \|a_k(0)\|_{\ell^1}^{1/2}\int_t^{t+1} \||k|^\beta a_k(s)\|_{\ell^1}^{1/2}\,ds 
    \notag\\&\leq  C
    e^{- C^{-1} t}  \|a_k(0)\|_{\ell^1}^{1/2} \Big( \int_t^{t+1} \sum_{k} |k|^{\beta} a_k(s)\,ds\Big)^{1/2}
    \notag\\&\leq   C
    e^{- C^{-1} t}  \|a_k(0)\|_{\ell^1}^{1/2} \Big( \int_t^{t+1} \sum_{k} \Big(\frac{S(R|k|)}{\log |k| + 1}\Big)^2a_k(s)\,ds\Big)^{1/2},
    \label{eq:exponential decay interpolated}
    \end{align}
    where we use~\eqref{eq:S not too small} of Assumption~\ref{asmp:wk}. Note that by~\eqref{eq:ell p energy equality} with $p=1$, we have that $\|a_k(t)\|_{\ell^1}$ is decreasing in $t$. Combining this fact with~\eqref{eq:exponential decay interpolated} and~\eqref{eq:smoothing}, we have that
    \[\|a_k(t+1)\|_{\ell^1}\leq    \int_t^{t+1} \|a_k(s)\|_{\ell^1}\,ds \leq  C
    e^{- C^{-1} t}  \|a_k(0)\|_{\ell^1}^{1/2} \|a_k(t)\|_{\ell^1}^{1/2} \leq C e^{-C^{-1}t }  \|a_k(0)\|_{\ell^1}.\]
    This then gives~\eqref{eq:exponential decay}, after perhaps increasing the constant $C$ to cover $t \in [0,1].$
\end{proof}

\section{Obukhov--Corrsin lower bounds}

\label{s:lower bounds}

We let $\eta \in C_c^\infty(B_1)$ be such that $\eta \geq 0$ and $\int \eta(x)\,dx =1.$ We let $\eta_r(x) :=r^{-d} \eta(x/r).$ For $r \in (0,1),$ we view $\eta_r$ as a function $\T^d \to \R$ under the identification $\T^d \cong [0,1]^d/ \sim$.

\subsection{Obukhov--Corrsin lower bounds for correlated-in-time models}

The following is essentially the argument of~\cite{constantin_onsagers_1994}, which appeared first in the passive scalar case in~\cite[Theorem 4]{drivas_anomalous_2022}.

\begin{proposition}
\label{prop:CET}
    Let $\phi^\kappa_t$ solve~\eqref{eq:advection-diffusion-free}. Then there exists $C(d)>0$ such that for all $r \in (0,1),$ we have the bound
      \begin{equation}
      \label{eq:CET} {-}\frac{d}{dt} \|\eta_r* \phi^\kappa_t\|_{L^2}^2 \leq C \big(\kappa r^{-2}  + \|v_t\|_{C^\alpha_x} r^{-(1-\alpha)} \big)  \frac{1}{|B_r|} \int_{B_r} \big\|\phi_t^\kappa(x-y) - \phi_t^\kappa(x)\big\|_{L^2_x}^2\,dy.
      \end{equation}
    Thus in particular, for all $r \geq \kappa^{\frac{1}{1+\alpha}}$, 
    \begin{equation}
        \label{eq:CET fourier}
         {-}\frac{d}{dt} \|\eta_r* \phi^\kappa_t\|_{L^2}^2 \leq C (1 + \|v_t\|_{C^\alpha_x}) r^{-(1-\alpha)}  \sum_{k \in \Z^d \backslash\{0\}} \min(|k|^2 r^2, 1) |\hat \phi^\kappa_t(k)|^2.
    \end{equation}
\end{proposition}

\begin{proof}
    Convolving~\eqref{eq:advection-diffusion-free} with $\eta_r$, we see that
    \[\partial_t (\eta_r * \phi^\kappa_t) = \kappa \Delta (\eta_r * \phi^\kappa_t)- (\eta_r * v_t)\cdot \nabla (\eta_r  *\phi^\kappa_t) + \nabla \cdot \big(\eta_r * v_t \eta_r * \phi^\kappa_t - \eta_r*(v_t \phi^\kappa_t)\big). \]
    Thus
    \begin{align} {-\frac{1}{2}}\frac{d}{dt} \|\eta_r* \phi^\kappa_t\|_{L^2}^2 &=  \kappa \|\nabla \eta_r * \phi^\kappa_t\|_{L^2}^2 +\int \nabla \eta_r * \phi_t^\kappa \big( \eta_r * v_t \eta_r * \phi^\kappa_t - \eta_r*(v_t \phi^\kappa_t)\big)
   \notag \\&\leq  \kappa \|\nabla \eta_r * \phi^\kappa_t\|_{L^2}^2 +\| \nabla \eta_r * \phi_t^\kappa\|_{L^2}  \big\| \eta_r * v_t \eta_r * \phi^\kappa_t - \eta_r*(v_t \phi^\kappa_t)\|_{L^2}.
   \label{eq:big display for CET}
    \end{align}
    We then note that
    \begin{align}\|\nabla \eta_r * \phi^\kappa_t\|_{L^2}& = \Big \|\int \nabla \eta_r(y) \phi^\kappa_t(x-y)\,dy\Big\|_{L^2_x} 
    \notag\\&= \Big \|\int \nabla \eta_r(y) \big(\phi^\kappa_t(x-y)-\phi^\kappa_t(x)\big)\,dy\Big\|_{L^2_x}
   \notag \\&\leq C r^{-1} \frac{1}{|B_r|} \int_{B_r} \big\|\phi_t^\kappa(x-y) - \phi_t^\kappa(x)\big\|_{L^2_x}\,dy,
   \label{eq:grad term}
    \end{align}
    where we use the definition of $\eta$ and Minkowski's integral inequality for the final line.

    Note that
    \begin{align*}&\eta_r * v_t(x) \eta_r * \phi^\kappa_t(x) - \eta_r*(v_t \phi^\kappa_t)(x)
    \\&\qquad= (\eta_r * v_t(x) - v_t(x))(  \eta_r * \phi^\kappa_t(x)  - \phi^\kappa_t(x))
    \\&\qquad\qquad - \int \eta_r(y) \big(v_t(x-y) - v_t(x)\big)\big(\phi^\kappa_t(x-y) - \phi^\kappa_t(x)\big)\,dy.
    \end{align*}
    Thus, again using Minkowski's integral inequality, we can readily see
    \begin{align}
         &\big\| \eta_r * v_t \eta_r * \phi^\kappa_t - \eta_r*(v_t \phi^\kappa_t)\|_{L^2} 
         \notag\\&\qquad\leq 2\sup_{|y| \leq r,x \in \T^d} |v_t(x-y) - v_t(x)| \int \eta_r(y)\big\|\phi^\kappa_t(x-y) - \phi^\kappa_t(x)\big\|_{L^2_x}\,dy
         \notag\\&\qquad\leq 2r^{\alpha} \|v_t\|_{C^\alpha_x} \frac{1}{|B_r|} \int_{B_r} \big\|\phi_t^\kappa(x-y) - \phi_t^\kappa(x)\big\|_{L^2_x}\,dy.
         \label{eq:commutator term}
    \end{align}
    Then combining~\eqref{eq:big display for CET},~\eqref{eq:grad term},~\eqref{eq:commutator term}, and Jensen's inequality, we get~\eqref{eq:CET}. For~\eqref{eq:CET fourier}, we use that $r \geq \kappa^{\frac{1}{1+\alpha}}$, the Plancherel isomorphism, and the behavior of Fourier coefficients under translation to give 
    \[ {-}\frac{d}{dt} \|\eta_r* \phi^\kappa_t\|_{L^2}^2 \leq C (1 + \|v_t\|_{C^\alpha_x}) r^{-(1-\alpha)} \sum_k |\hat \phi^\kappa_t|^2(k) \frac{1}{|B_r|}\int_{B_r} |1 - e^{2\pi i k \cdot y}|^2\,dy.\]
    Bounding the integral on the right hand side, we conclude.
\end{proof}

\begin{proof}[Proof of Theorem~\ref{thm:correlated lower bound}]
    Letting $\phi^\kappa_t$ solve~\eqref{eq:advection-diffusion-free}, we have that (both by the anomalous dissipation and by the fact $\kappa>0$),
    \[\lim_{t \to \infty} \E\|\phi^\kappa_t\|_{L^2}^2=0.\]
    Since $\eta_r*F$ converges to $F$ in $L^2(\T^d)$ as $r \to 0$, there exists some $r_0 >0$ such that for all $r \in (0,r_0)$, $\|\eta_r * \phi^\kappa_0\|_{L^2} = \|\eta_r * F\|_{L^2} \geq \frac{1}{2} \|F\|_{L^2}.$ Thus, taking an expectation, integrating~\eqref{eq:CET fourier} over $t \in [0,\infty)$, and using~\eqref{eq:v regularity}, for all $r \in (\kappa^{\frac{1}{1+\alpha}},r_0)$, we have that
    \begin{align}
    \tfrac{1}{4} \|F\|_{L^2}^2& \leq \lim_{t \to \infty} \E\|\eta_r * \phi^\kappa_0\|_{L^2}^2 - \E\|\eta_r * \phi^\kappa_t\|_{L^2}^2 
    \notag\\&\leq C r^{-(1-\alpha)} \E \int_0^\infty \sum_k \min(|k|^2 r^2, 1) |\hat \phi^\kappa_t(k)|^2
    \notag\\&=  C r^{-(1-\alpha)}\sum_k \min(|k|^2 r^2, 1)  \E_{\mu^\kappa} |\hat \theta^\kappa(k)|^2,
    \label{eq:lower bound integrated}
    \end{align}
    where $\mu^\kappa$ is the unique invariant measure from Proposition~\ref{prop:invariant measures} and we use the representation from Proposition~\ref{prop:invariant measures} for the final equality. Then for $1 \leq a < r^{-1} <  b < \infty$ to be chosen, we bound
    \begin{align}
         &\sum_k \min(|k|^2 r^2, 1)  \E_{\mu^\kappa} |\hat \theta^\kappa(k)|^2 
        \notag \\&\qquad\leq r^2\sum_{|k| < a} |k|^2 |\hat \theta^\kappa(k)|^2 + \sum_{a \leq |k| \leq b} \E_{\mu^\kappa} |\hat \theta^\kappa(k)|^2 + \sum_{|k| > b} \E_{\mu^\kappa} |\hat \theta^\kappa(k)|^2
      \notag   \\&\qquad\leq \sum_{a \leq |k| \leq b} \E_{\mu^\kappa} |\hat \theta^\kappa(k)|^2 +  r^2 a^{2(1-\beta)} (\log a +1)^{2m} \sum_{|k| < a} \frac{|k|^{2 \beta}}{(\log |k| +1)^{2m}} \E_{\mu^\kappa} |\hat \theta^\kappa(k)|^2 
       \notag  \\&\qquad\qquad\qquad+ C \frac{(\log |b| + 1)^{2m}}{b^{2\beta}}\sum_{|k| > b} \E_{\mu^\kappa}\frac{|k|^{2 \beta}}{(\log |k| +1)^{2m}} |\hat \theta^\kappa(k)|^2
        \notag \\&\qquad\leq \sum_{a \leq |k| \leq b} \E_{\mu^\kappa} |\hat \theta^\kappa(k)|^2 + C \Big(r^2 a^{2(1-\beta)} (\log a +1)^{2m}  + \frac{(\log |b| + 1)^{2m}}{b^{2\beta}}\Big)\|F\|_{L^2}^2,
         \label{eq:high low middle split}
    \end{align}
    where we use~\eqref{eq:sigma lower bound correlated}, the assumption of anomalous regularization up to $\sigma(H)$, and Proposition~\ref{prop:OC upper bounds} for the final inequality.

    Choosing
    \[a = C^{-1} \frac{r^{- \frac{1+\alpha}{2(1-\beta)}}}{\big(\log r^{-1}\big)^{\frac{m}{1-\beta}}} \quad \text{and} \quad b=  C r^{-\frac{1-\alpha}{2\beta}} \big(\log r^{-1}\big)^{\frac{m}{\beta}}\]
    for $C$ sufficiently large, we have that 
    \[ C r^{-(1-\alpha)}\Big(r^2 a^{2(1-\beta)} (\log a +1)^{2m}  + \frac{(\log |b| + 1)^{2m}}{b^{2\beta}}\Big)\|F\|_{L^2} \leq \frac{1}{8} \|F\|_{L^2}^2,\]
    so we can combine~\eqref{eq:lower bound integrated} and~\eqref{eq:high low middle split} to give~\eqref{eq:lower bound correlated general}. The first inequality of~\eqref{eq:lower bound correlated particular} then follows directly. For the second inequality of~\eqref{eq:lower bound correlated particular}, we compute using Proposition~\ref{prop:OC upper bounds} and~\eqref{eq:sigma lower bound correlated}:
    \begin{align*}
        &\sum_{C^{-1} r^{- 1}(\log r^{-1})^{-\frac{2m}{\alpha+1}} \leq |k| \leq C r^{-1} (\log r^{-1})^{\frac{2m}{1-\alpha}}} \E_{\mu^\kappa} |\hat \theta^\kappa(k)|^2
        \\&\qquad\qquad\qquad\qquad\qquad \leq C (\log r^{-1})^{\frac{4m}{1+\alpha}} r^{1-\alpha}  \sum_{k} \frac{|k|^{1-\alpha}}{(\log |k|+1)^{2m}}\E_{\mu^\kappa} |\hat \theta^\kappa(k)|^2
        \\&\qquad\qquad\qquad\qquad\qquad \leq C (\log r^{-1})^{\frac{4m}{1+\alpha}} r^{1-\alpha}  \E_{\mu^\kappa} \|\theta^\kappa\|_{\sigma(H)}^2 \\&\qquad\qquad\qquad\qquad\qquad \leq C (\log r^{-1})^{\frac{4m}{1+\alpha}} r^{1-\alpha} \|F\|_{L^2}^2,
    \end{align*}
    giving the claimed bound.
\end{proof}

\subsection{Obukhov--Corrsin lower bounds for white-in-time models}
\label{s:oc bounds white}

\begin{definition}
    We define the covariance matrix $D : \T^d \to \R^{d \times d}$ of $du_t$ by
    \[\E du^i_t(x) du_s^j(y) =: \delta(t-s) D_{ij}(x-y).\]
\end{definition}

\begin{lemma}
    We have the representation
    \[D(x) = \sum_k |w_k|^2 \cos(2\pi k \cdot x) \sum_{j=1}^{d-1} \mathrm{e}_{k,j} \otimes \mathrm{e}_{k,j}.\]
    We then have that there exist $C(d)>0$ such that for all $\alpha \in (0,1)$,
    \begin{equation}
    \label{eq:D estimate}
    |D(0) - D(x)| \leq C |x|^{2\alpha} \sum_k |k|^{2\alpha} |w_k|^2.
    \end{equation}
\end{lemma}

\begin{proof}
    The representation is a direct computation, using that the imaginary part cancels as the coefficients are purely real and even in $k$. Then we compute
    \begin{align*}
    |D(0) - D(x)| &\leq C\sum_k |w_k|^2 (1-\cos(2\pi k \cdot x))
    \\&\leq C\sum_k |w_k|^2 \min(|k|^2 |x|^2,1)
    \\&\leq C\max_k  \min(|k|^{2(1-\alpha)} |x|^2,|k|^{-2\alpha})  \sum_k |k|^{2\alpha} |w_k|^2
    \\&\leq C |x|^{2\alpha} \sum_k |k|^{2\alpha} |w_k|^2,
\end{align*}
as claimed.
\end{proof}

The following is a direct computation using stochastic calculus. See, e.g., \cite[Section 2.2]{rowan_anomalous_2024} for the derivation.
\begin{proposition}
\label{prop:g info}
    Let $\phi^\kappa_t$ solve~\eqref{eq:kraichnan-free}. Define for $r \in \T^d$
    \begin{equation}
    \label{eq:g def}
    g^\kappa_t(x) := \int \E \phi^\kappa_t(y) \phi^\kappa_t(y+x)\,dy.
    \end{equation}
    Then
    \begin{equation}
    \label{eq:g equation}
    \dot g^\kappa_t= 2 \kappa \Delta g + (D(0) - D(x)) : \nabla^2 g^\kappa_t.
    \end{equation}
\end{proposition}

\begin{proposition}
    Suppose that the $w_k$ satisfy~\eqref{eq:dut-regularity}, then there exists $C(d)>0$ such that for all $r \in (0,1),$ we have the bound
    \begin{equation}
    \label{eq:CET kraichnan}
    - \frac{d}{dt}\eta_r * g^\kappa_t(0)  \leq C (\kappa + r^{2\alpha}) r^{-2} \frac{1}{|B_r|}\int_{B_r} \E \|\phi^\kappa_t(x) - \phi^\kappa(x-y)\|_{L^2_x}^2\,dy.
\end{equation}
    Thus in particular, for all $r \in (\kappa^{\frac{1}{2\alpha}},1)$, 
    \begin{equation}
        \label{eq:CET kraichnan fourier}
        - \frac{d}{dt}\eta_r * g^\kappa_t(0)  \leq C r^{-2(1-\alpha)} \sum_{k \in \Z^d \backslash\{0\}} \min(|k|^2 r^2, 1) |\hat \phi^\kappa_t(k)|^2.
    \end{equation}
\end{proposition}

\begin{proof}
    By~\eqref{eq:g equation} and~\eqref{eq:g def},
    \begin{align}
        - \frac{d}{dt}\eta_r * g^\kappa_t(0) &=-\E \int \eta_r(x)\phi^\kappa_t(y)\big(( 2\kappa \Delta + (D(0) - D(x)): \nabla^2) \phi^\kappa_t\big)(y-x)\,dxdy
        \notag\\&= - \E \int (2\kappa I + (D(0) - D(x)):\nabla^2\eta_r(x)\phi^\kappa_t(y) \phi^\kappa_t(y-x)\,dxdy,
        \notag
    \end{align}
    where in order to move the derivatives in $x$, we use that $\sum_{i=1}^d \partial_i D_{ij} =  \sum_{j=1}^d \partial_j D_{ij} = 0$ by construction (this is a version of the fact that $\nabla \cdot du_t =0$ by construction). We then note that 
    \begin{align*}&- \E \int (2\kappa I + (D(0) - D(x)):\nabla^2\eta_r(x)\phi^\kappa_t(y) \phi^\kappa_t(y-x)\,dxdy 
    \\&\qquad= \frac{1}{2}\E \int (2\kappa I + (D(0) - D(x)):\nabla^2\eta_r(x)(\phi^\kappa_t(y)- \phi^\kappa_t(y-x))^2\,dxdy,
    \end{align*}
    since the $\phi^\kappa_t(y)^2$ term is killed by integrating by parts the $\nabla_x$ on $\eta_r$ and then $\phi^\kappa_t(y-x)^2$ term is equal to $\phi^\kappa_t(y)^2$ term by changing variables in $y$. Combining the two displays above, using the definition of $\eta_r$,~\eqref{eq:D estimate}, and~\eqref{eq:C alpha condition}, we get~\eqref{eq:CET kraichnan}. \eqref{eq:CET kraichnan fourier} then follows as in Proposition~\ref{prop:CET}.
\end{proof}

\begin{proof}[Proof of Theorem~\ref{thm:white lower bound}]
    We note that for $x \in \T^d$, by the Cauchy-Schwarz inequality,
    \[g^\kappa_t(x) = \E\int  \phi^\kappa_t(y) \phi^\kappa_t(y+x)\,dy \leq \E \|\phi^\kappa_t\|_{L^2}^2.\]
    Then---both since $\kappa >0$ and due to the anomalous dissipation---we have that as $t \to \infty$, $\E \|\phi^\kappa_t\|_{L^2}^2 \to 0$. Thus for any $r \in (0,1)$, as $t \to \infty$, $\eta_r * g^\kappa_t(0) \to 0$. Also, by the definition of $g^\kappa$, we have that $g^\kappa_0 \in C^0(\T^d)$ since $\phi^\kappa_0 = F \in L^2(\T^d)$. Thus $\eta_r * g^\kappa_0(0) \to g^\kappa_0(0) = \|F\|_{L^2}^2$ as $r \to 0$, so there exists some $r_0 \in (0,1)$ such that for all $r \leq r_0$, $\eta_r * g^\kappa_0(0) \geq \frac{1}{2} \|F\|_{L^2}^2.$
    
    Combining this with~\eqref{eq:CET kraichnan fourier}, by integrating over $t \in [0,\infty),$ we have for any $r \in (\kappa^{\frac{1}{2\alpha}}, 1)$, 
    \begin{align}\tfrac{1}{2} \|F\|_{L^2}^2 &\leq C r^{-2(1-\alpha)} \sum_{k \in \Z^d \backslash\{0\}} \min(|k|^2 r^2, 1) \int_0^\infty |\hat \phi^\kappa_t(k)|^2\,dt 
    \notag\\&=  C r^{-2(1-\alpha)} \sum_{k \in \Z^d \backslash\{0\}} \min(|k|^2 r^2, 1) \E_{\mu^\kappa} |\hat \theta^\kappa_t(k)|^2,
    \label{eq:lower bound integrated kraichnan}
    \end{align}
    where for the equality we use the representation from Proposition~\ref{prop:invariant measures} for $\mu^\kappa$ the unique invariant measure of~\eqref{eq:kraichnan-forced}. Then, for $1 \leq a < |k| < b < \infty$ to be chosen, we have that
    \begin{align}
    &\sum_{k} \min(|k|^2 r^2, 1) \E_{\mu^\kappa} |\hat \theta^\kappa_t(k)|^2 
    \notag\\&\qquad\leq  \sum_{a\leq  |k| \leq b} \E_{\mu^\kappa} |\hat \theta^\kappa_t(k)|^2  + 
   r^2 \sum_{|k| < a} \E_{\mu^\kappa} |k|^2 |\hat \theta^\kappa_t(k)|^2  + \sum_{|k| > b} \E_{\mu^\kappa} |\hat \theta^\kappa_t(k)|^2
   \notag\\& \qquad\leq  \sum_{a\leq |k| \leq b} \E_{\mu^\kappa} |\hat \theta^\kappa_t(k)|^2  + 
   r^2 a^{2(1-\beta)} (\log a + 1)^{2m} \sum_{|k| < a} \E_{\mu^\kappa} \frac{|k|^{2\beta}}{(\log |k| + 1)^{2m}} |\hat \theta^\kappa_t(k)|^2  
  \notag \\&\qquad \qquad\qquad+ C\frac{(\log b + 1)^{2m}}{b^{2\beta}} \sum_{|k| > b}  \frac{|k|^{2\beta}}{(\log |k| + 1)^{2m}} \E_{\mu^\kappa} |\hat \theta^\kappa_t(k)|^2
  \notag  \\&\qquad\leq \sum_{a\leq |k| \leq b} \E_{\mu^\kappa} |\hat \theta^\kappa_t(k)|^2 + C \Big( r^2 a^{2(1-\beta)} (\log a + 1)^{2m} + \frac{(\log b + 1)^{2m}}{b^{2\beta}}\Big) \|F\|_{L^2},
    \label{eq:high low middle split kraichnan}
    \end{align}
    where we use~\eqref{eq:sigma lower bound kraichnan}, the assumption of anomalous regularization up to $\sigma(H)$, and Proposition~\ref{prop:OC upper bounds} for the final inequality.

    Choosing
    \[a = C^{-1} \frac{r^{-\frac{\alpha}{1-\beta}}}{(\log r^{-1})^{\frac{m}{1-\beta}}} \quad \text{and} \quad b= C r^{-\frac{1-\alpha}{\beta}} (\log r^{-1})^{\frac{m}{\beta}},\]
    for $C$ sufficiently large, we have that 
    \[ C r^{-2(1-\alpha)}\Big(r^2 a^{2(1-\beta)} (\log a +1)^{2m}  + \frac{(\log |b| + 1)^{2m}}{b^{2\beta}}\Big)\|F\|_{L^2} \leq \frac{1}{4} \|F\|_{L^2}^2,\]
    so we can combine~\eqref{eq:lower bound integrated kraichnan} and~\eqref{eq:high low middle split kraichnan} to give~\eqref{eq:lower bound white general}. The first inequality of ~\eqref{eq:lower bound white particular} then follows directly. For the second inequality of~\eqref{eq:lower bound white particular}, we compute using Proposition~\ref{prop:OC upper bounds} and~\eqref{eq:sigma lower bound kraichnan}:
    \begin{align*}
        &\sum_{ C^{-1} r^{-1}(\log r^{-1})^{-\frac{m}{\alpha}}  \leq |k| \leq C r^{-1} (\log r^{-1})^{\frac{m}{1-\alpha}}} \E_{\mu^\kappa} |\hat \theta^\kappa(k)|^2 
        \\&\qquad\qquad\qquad\qquad\qquad\leq  C (\log r^{-1})^{\frac{2m}{\alpha}} r^{2(1-\alpha)}  \E_{\mu^\kappa} \|\theta^\kappa\|_{\sigma(H)}^2
        \\&\qquad\qquad\qquad\qquad\qquad\leq  C (\log r^{-1})^{\frac{2m}{\alpha}} r^{2(1-\alpha)} \|F\|_{L^2}^2,
    \end{align*}
    giving the claimed bound.
\end{proof}

\section{Weighted lattice Poincar\'e inequalities}
\label{s:lattice}

In this section, we fix $\alpha \in (0,1)$ and coefficients $(w_k)_{k \in \Z^d \backslash \{0\}}$ satisfying Assumption~\ref{asmp:wk}. We note that
\begin{equation}
    \label{eq:S bar bound}
    S(r) \leq r^{1-\alpha} \sum_{j} |j|^{2\alpha} w_j^2 \leq r^{1-\alpha}.
\end{equation}

The goal of this section is to prove Proposition~\ref{prop:ell p inequality clean}, which will be a direct corollary of Lemma~\ref{lem:ell p type inequality with explicit constants}. We keep explicit track of constants since we will need to choose some constants sufficiently large compared to other constants, and the validity of this argument is made clearest by keeping constants explicit.

Our first step will be proving the following inequality of $\ell^1$-type. While we ultimately want the inequality of Proposition~\ref{prop:ell p inequality clean}, which is an $\ell^p$-type inequality, we will see this $\ell^1$-type inequality will imply the $\ell^p$-type inequality. We can view this $\ell^1$-type inequality as a weighted Poincar\'e inequality on the lattice, since the right hand side involves only differences $|a_{k+j} - a_k|$ and there is weights on both sides (note however these weights are not matching: the weight on the left is generically much larger than the weight on the right, giving us a substantial ``gain'' in the weighting). The idea behind the argument is first to break up into geometric annuli, which is helpful since on a geometric annulus the weights are pointwise comparable up to a uniform constant. Then we want to use the ``fundamental theorem of calculus'' to write $a_k$ in terms of differences $a_{k+mj+j} - a_{k+mj}$ and an endpoint $a_{k+nj}$. We will then utilize that $w_j>0$ for infinitely many $j$, so on larger annuli, there are more and more differencing directions $j$ we can take advantage of. This growth of differencing directions is what is ultimately responsible for the gain in the weight in the inequality. 

\begin{lemma}
\label{lem:ell 1 type inequality}
	Fix $\alpha \in (0,1)$ and $(w_k)_{k \in \Z^d \backslash \{0\}}$ satisfying Assumption~\ref{asmp:wk} for $\alpha$. Then,  for all $R \geq r_0$, $(a_k)_{k \in \Z^d \backslash \{0\}}$ such that $a_k \geq 0 $ and 
        \begin{equation}
        \label{eq:regularity assumption}
        \sum_{k \in \Z^d \backslash\{0\}} |k|^{2(1-\alpha)} a_k <\infty,
    \end{equation}
    for $S$ defined by~\eqref{eq:S-def}, we have the bound
    \begin{align*}&\sum_{k \in \Z^d \backslash \{0\}} S(R |k|)^2 a_k 
    \\&\qquad\leq 192 R^2 \delta^{-2}\sum_{k \in \Z^d \backslash\{0\}}\sum_{\substack{j \in \Z^d \backslash\{0\}\\ |k+j| \geq (24 R \delta^{-1})^{-2} |k| }} |j|^{\alpha} w_j^2  S(24 \delta^{-1} R^2|k|) |\Pi_{j^\perp} k||a_{k+j} - a_k|.
    \end{align*}
\end{lemma}

\begin{proof}
    Let $\lambda = 24 \delta^{-1} R$.  Then for each $n \in \N$, let $A_n := \{k \in \Z^d : \lambda^{n-1} < |k| \leq \lambda^n\}.$
    
    We first note that for any $k, j \in \Z^d \backslash \{0\}$ and $b \geq 1$, we have that
    \[a_k = a_{k+b j} - \sum_{m=0}^{b-1} a_{k+m j + j} - a_{k+m j}.\]
    We now fix $k \in A_n$. Then, applying the triangle inequality and summing $b \in \N$ such that $k+b j \in A_{n+1}$, we see that
    \begin{align*} &|\{b \in \N : k+b j \in A_{n+1}\}| a_k \leq \sum_{b \geq 0, k+b j \in A_{n+1}} \Big(a_{k+b j} +  \sum_{m=0}^{b-1}| a_{k+m j + j} - a_{k+m j}|\Big)
    \\&\quad\leq \sum_{b \geq 0, k+b j \in A_{n+1}} a_{k+b j} +  |\{b \in \N : k+b j \in A_{n+1}\}|  \sum_{\substack{m \geq 0 \\|k + mj +j| \leq \lambda^{n+1}}} | a_{k+m j + j} - a_{k+m j}|.
    \end{align*}
    Dividing, we thus get
    \[a_k \leq \frac{1}{|\{b \in\N  : k+b j \in A_{n+1}\}|}\sum_{b \geq 0, k+b j \in A_{n+1}} a_{k+b j} + \sum_{\substack{m \geq 0 \\|k + mj +j| \leq \lambda^{n+1}}} | a_{k+m j + j} - a_{k+m j}|.\]
    We only want the differences $|a_{k+mj+j} - a_{k+mj}|$ to appear when $k+mj$ is large, that is $|k+mj| > \lambda^{n-1}$. For an arbitrary $j$, $k+mj$ may get close to $0$ for $m \geq0$, but in that case we can consider $-j$ instead. That is, we note that either for all $m \geq 0, |k+mj| \geq |k|$ or for all $m \geq 0, |k-mj| \geq |k|$. Thus for $\tilde j =j$ in the first case or $\tilde j = -j$ in the second, applying the above inequality for $\tilde j$ gives
    \begin{align*}a_k &\leq \frac{1}{|\{b \in \N : k+b\tilde j \in A_{n+1}\}|}\sum_{b \geq 0, k+b \tilde j \in A_{n+1}} a_{k+b \tilde j} + \sum_{\substack{m \geq 0\\  |k + m\tilde j +\tilde j| \leq \lambda^{n+1}}} | a_{k+m \tilde j + \tilde j} - a_{k+m \tilde j}|
    \\&\leq \frac{2}{|(k + \<j\>) \cap A_{n+1}|}\sum_{\ell \in (k + \<j\>) \cap A_{n+1}} a_{\ell} + \sum_{\substack{\ell, \ell+j \in (k + \<j\>) \cap (A_{n} \cup A_{n+1})}} |a_{\ell+j} - a_\ell|.
    \end{align*}
     We now sum this inequality over $k \in A_n, |j| \leq R \lambda^n$---using that since $\lambda \geq 8R$ and $R \geq 2$, for any such $k,j$ we have that $(k+ \<j\>) \cap A_{n+1} \ne \emptyset$. In the sum, we add the $n,k,j$ dependent weight $\nu^n_j \frac{|\Pi_{j^\perp} k|}{|k|}$---with $\nu^n_j \geq 0$ specified below---giving:
    \begin{align}&\sum_{|j| \leq R \lambda^n} \sum_{k \in A_n} \nu^n_j \frac{|\Pi_{j^\perp} k|}{|k|}  a_k 
    \notag\\&\quad\qquad\leq \sum_{k \in A_n} \sum_{|j| \leq R \lambda^n} \frac{2 \nu^n_j |\Pi_{j^\perp} k|}{|k| |(k + \<j\>) \cap A_{n+1}|}\sum_{\ell \in (k + \<j\>) \cap A_{n+1}} a_{\ell} 
    \notag\\&\quad\qquad\qquad+ \sum_{k \in A_n}\sum_{|j| \leq R \lambda^n} \nu^n_j \frac{|\Pi_{j^\perp} k|}{|k|}\sum_{\ell,\ell+j \in (k + \<j\>) \cap (A_{n} \cup A_{n+1})} |a_{\ell+j} - a_\ell|\notag.
\end{align}
 Using then that $\ell - k \in \<j\>$, so $\ell + \<j\> = k + \<j\>$ and $\Pi_{j^\perp} k = \Pi_{j^\perp} \ell$, we get that
\begin{align}
&\sum_{|j| \leq R \lambda^n} \sum_{k \in A_n} \nu^n_j \frac{|\Pi_{j^\perp} k|}{|k|}  a_k 
    \notag\\&\quad\qquad\leq 2  \sum_{|j| \leq R \lambda^n} \sum_{\ell \in A_{n+1}}\sum_{k \in (\ell + \<j\>) \cap A_n}\frac{ \nu^n_j}{ |(\ell + \<j\>) \cap A_{n+1}|} a_{\ell}  
    \notag\\&\quad\qquad\qquad+ \lambda^{1-n}\sum_{|j| \leq R \lambda^n} \sum_{\ell,\ell+j \in A_{n} \cup A_{n+1}} \sum_{k \in (\ell + \<j\>) \cap A_n} \nu^n_j |\Pi_{j^\perp} \ell|  |a_{\ell+j} - a_\ell|
    \notag\\&\quad\qquad=   2\sum_{j \leq R\lambda^n }\sum_{\ell \in A_{n+1}}   \nu^n_j \frac{ |(\ell + \<j\>) \cap A_{n}|}{|(\ell + \<j\>) \cap A_{n+1}|} a_{\ell} 
        \notag\\&\quad\qquad\qquad+ \lambda^{1-n} \sum_{|j| \leq R \lambda^n}\sum_{\ell,\ell+j \in A_{n} \cup A_{n+1}} \nu^n_j | (\ell + \<j\>) \cap A_n ||\Pi_{j^\perp} \ell||a_{\ell+j} - a_\ell|.
        \label{eq:main lattice sum bound}
\end{align}
 We now note that  $| (\ell + \<j\>) \cap A_m| \leq \frac{2\lambda^m}{|j|} + 1$ and for $\ell \in A_{n+1}$ such that $|(\ell +\<j\> \cap A_n| >0$, we have that
\[|(\ell + \<j\>) \cap A_{n+1}| \geq \frac{\lambda^{n+1} - \lambda^n}{|j|} -1 \geq \frac{\lambda^{n+1}}{2|j|} - 1\]
Then, for $\ell \in A_{n+1}, |j| \leq R\lambda^n$, we have that 
\begin{equation}
\label{eq:ratio bound}    
\frac{ |(\ell + \<j\>) \cap A_{n}|}{|(\ell + \<j\>) \cap A_{n+1}|}  \leq \frac{4 \lambda^n  +2|j|}{\lambda^{n+1} - 2|j|} \leq \frac{4 \lambda^n + 2 R \lambda^n}{\lambda^{n+1} - 2 R \lambda^n} \leq 6R\lambda^{-1},
\end{equation}
using that $R \geq 4$ and $\lambda \geq 4R$. We also have that for $|j| \leq  R \lambda^n$,
\begin{equation}
    \label{eq:set size bound}
    | (\ell + \<j\>) \cap A_n | \leq \frac{2\lambda^n}{|j|} + 1 \leq 2R \frac{\lambda^n}{|j|}.
\end{equation}

Combining~\eqref{eq:main lattice sum bound},~\eqref{eq:ratio bound}, and~\eqref{eq:set size bound} and summing over $n$, we have that
\begin{align}
    \sum_{n=1}^\infty \sum_{|j| \leq R \lambda^n} \sum_{k \in A_n} \nu^n_j \frac{|\Pi_{j^\perp} k|}{|k|} a_k & \leq 12 R\lambda^{-1}\sum_{n=1}^\infty \sum_{|j| \leq R \lambda^n}\sum_{k \in A_{n+1}}  \nu^{n}_ja_{k} 
    \notag\\&\quad+ 2 R \lambda \sum_{n=1}^\infty \sum_{|j| \leq R \lambda^n}\sum_{k, k+j\in A_{n} \cup A_{n+1}}  \frac{\nu^n_j}{|j|} |\Pi_{j^\perp} k||a_{k+j} - a_k|.
     \label{eq:summed in n bound 1}
\end{align}
We now choose $\nu^n_j := |j|^{1+\alpha} w_j^2  S(R\lambda^n)$, giving that 
\begin{align}
   &\sum_{n=1}^\infty  \sum_{k \in A_n}  S(R\lambda^n)^2 a_k 
   \notag\\&\qquad\leq \delta^{-1}\sum_{n=1}^\infty  \sum_{k \in A_n} S(R\lambda^n)  a_k\sum_{|j| \leq R \lambda^n} |j|^{1+\alpha} w_j^2 \frac{|\Pi_{j^\perp} k|}{|k|} 
    \notag\\&\qquad \leq 12 \delta^{-1} R\lambda^{-1}\sum_{n=1}^\infty \sum_{|j| \leq R \lambda^n}\sum_{k \in A_{n+1}} |j|^{1+\alpha} w_j^2  S(R\lambda^n)a_{k} 
    \notag\\&\qquad\qquad+ 2 R \lambda \delta^{-1} \sum_{n=1}^\infty \sum_{|j| \leq R \lambda^n}\sum_{k, k+j\in A_{n} \cup A_{n+1}}  |j|^{\alpha} w_j^2  S(R\lambda^n) |\Pi_{j^\perp} k||a_{k+j} - a_k|,
    \label{eq:summed in n bound 2}
\end{align}
where we use the definition of $S$~\eqref{eq:S-def} as well as the assumption~\eqref{eq:nondegenerate}---which applies as $R\lambda^n \geq R \geq r_0$---for the first inequality and~\eqref{eq:summed in n bound 1} for the second inequality. Then again using the definition of $S$ together with $\lambda = 24 R \delta^{-1}$, we see that 
\begin{align} 12 \delta^{-1} R\lambda^{-1}\sum_{n=1}^\infty \sum_{|j| \leq R \lambda^n}\sum_{k \in A_{n+1}} |j|^{1+\alpha} w_j^2  S(R\lambda^n)a_{k} &\leq \tfrac{1}{2} \sum_{n=1}^\infty \sum_{k \in A_{n+1}}  S(R\lambda^n)^2a_{k} 
\notag\\&\leq \tfrac{1}{2} \sum_{n=1}^\infty \sum_{k \in A_n} S(R\lambda^n)^2 a_k.
\label{eq:term to reabsorb}
\end{align}
We also have from~\eqref{eq:regularity assumption} and~\eqref{eq:S bar bound},
\begin{align*} \sum_{n=1}^\infty \sum_{k \in A_n} S(R\lambda^n)^2 a_k &\leq (\lambda R)^{2(1-\alpha)} \sum_{n=1}^\infty \sum_{k \in A_n} (\lambda^{n-1})^{2(1-\alpha)} a_k \\&\leq (\lambda R)^{2(1-\alpha)} \sum_{n=1}^\infty \sum_{k \in A_n} |k|^{2(1-\alpha)} a_k < \infty.
\end{align*}
Together with~\eqref{eq:term to reabsorb}, this gives that we can reabsorb the first term on the right hand side of~\eqref{eq:summed in n bound 2} and use the definition of $\lambda$ to get
\begin{align}
   &\sum_{n=1}^\infty  \sum_{k \in A_n}  S(R\lambda^n)^2 a_k \notag\\&\qquad\leq 96 R^2 \delta^{-2}\sum_{n=1}^\infty \sum_{|j| \leq R \lambda^n}\sum_{k, k+j\in A_{n} \cup A_{n+1}}  |j|^{\alpha} w_j^2  S(R\lambda^n) |\Pi_{j^\perp} k||a_{k+j} - a_k|.
    \label{eq:summed in n bound 3}
\end{align}
Using that $S$ is increasing and that for $k \in A_n, |k| \leq \lambda^n,$ we have that 
\begin{equation}
    \label{eq:lhs simplification}
\sum_{k \in \Z^d \backslash \{0\}} S(R |k|)^2 a_k= \sum_{n=1}^\infty  \sum_{k \in A_n}  S(R|k|)^2 a_k  \leq \sum_{n=1}^\infty  \sum_{k \in A_n}  S(R\lambda^n)^2 a_k.
\end{equation}
Using that for $k, k+j \in A_n \cup A_{n+1}$, $\lambda^{n-1} \leq |k| \leq \lambda^{n+1}$ and $|k+j| \geq \lambda^{n-1}$, we also have that
\begin{align}
    &\sum_{n=1}^\infty \sum_{|j| \leq R \lambda^n}\sum_{k, k+j\in A_{n} \cup A_{n+1}}  |j|^{\alpha} w_j^2  S(R\lambda^n) |\Pi_{j^\perp} k||a_{k+j} - a_k|
   \notag\\ &\qquad\qquad\leq  2 \sum_{k \in \Z^d \backslash\{0\}}\sum_{\substack{j \in \Z^d \backslash\{0\}\\ |k+j| \geq \lambda^{-2} |k| }} |j|^{\alpha} w_j^2  S(R\lambda |k|) |\Pi_{j^\perp} k||a_{k+j} - a_k|.
   \label{eq:rhs simplification}
\end{align}
Finally, combining~\eqref{eq:summed in n bound 3},~\eqref{eq:lhs simplification},~\eqref{eq:rhs simplification}, and the definition of $\lambda$, we conclude.
\end{proof}

The following inequality is used in a similar setting in \cite[Proof of Lemma 2.6]{luo_elementary_2024}.
\begin{lemma}
    For $a,b \geq 0$ and $p>1$
    \begin{equation}
        \label{eq:bound on p difference}
    |a^p - b^p| \leq \frac{p}{\sqrt{p-1}} \sqrt{a^p + b^p} \sqrt{(a^{p-1}- b^{p-1})(a-b)}.
    \end{equation}
\end{lemma}

\begin{proof}
    Recall that for convex functions $f : [0,\infty) \to \R$, we have that for $x,y \in [0,\infty)$,
    \[f(y) \geq f(x) + f'(x) (y-x).\]
    Using that for $p>1$, $x^p$ and $x^{\frac{p}{p-1}}$ are convex, we have that
    \begin{align*}
           p b^{p-1} (a-b) &\leq a^p - b^p \leq p a^{p-1} (a-b) ,
           \\\frac{p}{p-1} b (a^{p-1} - b^{p-1})&\leq a^p- b^p \leq \frac{p}{p-1} a(a^{p-1} - b^{p-1}).
    \end{align*}
    Taking the product of the two inequality and taking a square root, we get the desired result.  
\end{proof}

\begin{lemma}
    \label{lem:ell p type inequality with explicit constants}
     Fix $\alpha \in (0,1)$ and $(w_k)_{k \in \Z^d \backslash \{0\}}$ satisfying Assumption~\ref{asmp:wk} for $\alpha$. Then, for all $R \geq r_0$, $(a_k)_{k \in \Z^d \backslash \{0\}}$ such that $a_k \geq 0 $ and 
        \begin{equation}
        \label{eq:regularity assumption ell p}
        \sum_{k \in \Z^d \backslash\{0\}} |k|^{2(1-\alpha)} a_k^p <\infty,
    \end{equation}
    for $S$ defined by~\eqref{eq:S-def}, we have the bound
    \begin{align*}
    &\sum_{k \in \Z^d \backslash \{0\}} S(R |k|)^2 a_k^p 
    \\&\qquad\leq  \frac{2^{17} R^2p^2 \Psi\big((24\delta^{-1} R)^3\big)^2}{\delta^2 (p-1)} \sum_{k,j \in \Z^d \backslash\{0\}} w_j^2  |\Pi_{j^\perp} k|^2(a_{k+j}^{p-1} - a_k^{p-1})(a_{k+j}- a_k).
    \end{align*}
\end{lemma}

\begin{proof}
    Applying Lemma~\ref{lem:ell 1 type inequality} to $a_k^p$, we have that
    \begin{align}&\sum_{k \in \Z^d \backslash \{0\}} S(R |k|)^2 a_k^p
    \notag\\&\qquad\leq 192 R^2 \delta^{-2}\sum_{k \in \Z^d \backslash\{0\}}\sum_{\substack{j \in \Z^d \backslash\{0\}
    \notag\\ |k+j| \geq (24 R \delta^{-1})^{-2} |k| }} |j|^{\alpha} w_j^2  S(24 \delta^{-1} R^2|k|) |\Pi_{j^\perp} k||a_{k+j}^p - a_k^p|
    \notag\\&\qquad\leq  \frac{192 R^2p}{\delta^2 \sqrt{p-1}} \sum_{k \in \Z^d \backslash\{0\}}\sum_{\substack{j \in \Z^d \backslash\{0\}\\ |k+j| \geq (24 R \delta^{-1})^{-2} |k| }} |j|^{\alpha} w_j^2  S(24 \delta^{-1} R^2|k|) |\Pi_{j^\perp} k|
    \notag\\&\qquad\qquad\qquad \times \sqrt{a_{k+j}^p + a_k^p}\sqrt{(a_{k+j}^{p-1} - a_k^{p-1})(a_{k+j}- a_k)}
    \notag\\&\qquad\leq  \frac{192 R^2p}{\delta^2 \sqrt{p-1}} \Big(\sum_{k \in \Z^d \backslash\{0\}}\sum_{\substack{j \in \Z^d \backslash\{0\}\\ |k+j| \geq (24 R \delta^{-1})^{-2} |k| }} |j|^{2\alpha} w_j^2 S(24 \delta^{-1} R^2|k|)^2 (a_{k+j}^p + a_k^p)\Big)^{1/2} 
    \notag\\&\qquad\qquad\qquad \times \Big(\sum_{k \in \Z^d \backslash\{0\}}\sum_{\substack{j \in \Z^d \backslash\{0\}\\ |k+j| \geq (24 R \delta^{-1})^{-2} |k| }} w_j^2  |\Pi_{j^\perp} k|^2(a_{k+j}^{p-1} - a_k^{p-1})(a_{k+j}- a_k)\Big)^{1/2},
    \label{eq:big cauchy schwarz display}
    \end{align}
    where we use~\eqref{eq:bound on p difference} for the second inequality and Cauchy-Schwarz for the third. We then use that $S$ is increasing and~\eqref{eq:C alpha condition} to bound
    \begin{align*}
        &\sum_{k \in \Z^d \backslash\{0\}}\sum_{\substack{j \in \Z^d \backslash\{0\}\\ |k+j| \geq (24 R \delta^{-1})^{-2} |k| }} |j|^{2\alpha} w_j^2 S(24 \delta^{-1} R^2|k|)^2 (a_{k+j}^p + a_k^p)
        \\&\qquad\leq  \sum_{k,j \in \Z^d \backslash\{0\}}|j|^{2\alpha} w_j^2  \Big(S\big((24 \delta^{-1} R)^3 R |k+j|\big)^2a_{k+j}^p + S(24 \delta^{-1} R^2|k|)a_k^p\Big)
        \\&\qquad\leq 2 \sum_{k} S\big((24 \delta^{-1} R)^3 R |k|\big)^2 a_k^p \sum_j  |j|^{2\alpha} w_j^2
        \\&\qquad\leq  2 \sum_{k} S\big((24 \delta^{-1} R)^3 R |k|\big)^2 a_k^p.
    \end{align*}
    We then use that $R \geq r_0$, allowing us apply~\eqref{eq:geometric fluctuation bound} to give that 
    \begin{align}
    \notag
    &\sum_{k \in \Z^d \backslash\{0\}}\sum_{\substack{j \in \Z^d \backslash\{0\}\\ |k+j| \geq (24 R \delta^{-1})^{-2} |k| }} |j|^{2\alpha} w_j^2 S(24 \delta^{-1} R^2|k|)^2 (a_{k+j}^p + a_k^p) 
    \\&\qquad\qquad\qquad\qquad\qquad\qquad\qquad\qquad\leq  2 \Psi((24 \delta^{-1} R)^3)^2\sum_{k} S(R |k|)^2 a_k^p.
    \label{eq:after geometric fluctuation bound}
    \end{align}
    We then note that from~\eqref{eq:S bar bound} and~\eqref{eq:regularity assumption ell p}, we have that 
    \[\sum_{k} S(R |k|)^2 a_k^p \leq R^{2(1-\alpha)} \sum_{k} |k|^{2(1-\alpha)} a_k^p <\infty.\]
    Thus, using~\eqref{eq:after geometric fluctuation bound}, we can factor out the first term on the right hand side of~\eqref{eq:big cauchy schwarz display} and take the square to give the result.
\end{proof}

{\small
\bibliographystyle{alpha}
\bibliography{keefer-references}
}

\end{document}